\documentclass[10pt, dvipsnames]{amsart}
\usepackage{amsmath}
\usepackage{amsfonts}
\usepackage{amssymb}
\usepackage{amsthm}
\usepackage{mathrsfs}       
\usepackage{comment}
\usepackage{graphicx}  
\usepackage{todonotes}
\usepackage{latexsym}
\usepackage{enumitem}
\newtheorem{problem}{PROBLEM}
\usepackage[T1]{fontenc}
\usepackage[utf8]{inputenc}
\allowdisplaybreaks
\usepackage[left=3cm,top=3cm,right=3cm,bottom=3cm,bindingoffset=0.5cm]{geometry}
\usepackage{xspace}
\usepackage{hyperref}

\usepackage{tikz}
\usepackage{tikz-cd}
\usetikzlibrary{decorations.markings}

\newcommand{\F}{\mathcal{F}}

\newcommand{\R}{\mathbb{R}}
\newcommand{\N}{\mathbb{N}}
\newcommand{\T}{\mathbb{T}}
\newcommand{\Z}{\mathbb{Z}}
\newcommand{\Q}{\mathbb{Q}}

\newcommand{\C}{\mathbb{C}}

\newcommand{\gl}{\textnormal{GL}}
\renewcommand{\L}{\mathcal{L}}

\numberwithin{equation}{section}
\renewcommand{\H}{\mathcal{H}}

\let \Re \relax
\DeclareMathOperator{\Re}{Re}
\let \Im \relax
\DeclareMathOperator{\Im}{Im}
\newcommand{\M}{\mathcal{M}}

\renewcommand{\d}{\partial}

\let \div \relax
\DeclareMathOperator{\div}{div}

\newcommand{\nocontentsline}[3]{}
\newcommand{\tocless}[2]{\bgroup\let\addcontentsline=\nocontentsline#1{#2}\egroup}

\newtheorem{thm}{THEOREM}[section]
\newtheorem{remark}[thm]{REMARK}
\newtheorem{lem}[thm]{LEMMA}
\newtheorem{defn}[thm]{DEFINITION}
\newtheorem{prop}[thm]{PROPOSITION}
\newtheorem{cor}[thm]{COROLLARY}

\newtheorem{hyp}{Assumption}
\newcounter{thmbiss}

\renewcommand{\P}{\mathbb{P}}

\def\<{\langle}
\def\>{\rangle}

\newcommand{\scalp}[2]{\langle #1 \, , #2\rangle}

\newcommand{\ml}{m(\lambda)}

\title[F-equivalence for parabolic systems and applications to nonlinear PDE]{F-equivalence for parabolic systems and applications to the stabilization of nonlinear PDE}
\author{Vincent Boulard and Amaury Hayat}

\usepackage{geometry} 
\begin{document}
\begin{abstract}
We consider the $F$-equivalence problem for parabolic systems: under which conditions a control system, governed by a parabolic operator $A$ and a control operator $B$, can be made equivalent to \textcolor{black}{an exponentially stable system with arbitrarily large decay rate} 
 through an appropriate control feedback law? While this problem has been resolved for finite-dimensional systems fifty years ago, good conditions for infinite-dimensional systems remain a challenge, especially for systems in spatial dimension larger than one.  Our main result establishes optimal conditions for the existence of an $F$-equivalence pair $(T,K)$ for a given parabolic control system $(A,B)$. We introduce an extended framework for $F$-equivalence of parabolic operators, addressing key limitations of existing approaches, and we prove that the pair $(T,K)$ is unique if and only if $(A,B)$ is approximately controllable. As a consequence, this provides a method to construct feedback operators for the rapid stabilization of nonlinear parabolic systems, possibly multidimensional in space. We provide several illustrative examples, including the rapid stabilization of the heat equation, the Kuramoto-Sivashinsky equation, the Navier-Stokes equations and the quasilinear heat equation.
\\



\end{abstract}
\maketitle
\tableofcontents

\setcounter{tocdepth}{2}

\section{Introduction}

{\color{black}
\noindent Consider the following nonlinear control problem
\begin{equation}
\label{eq:sys0}
\partial_{t}u(t) = Au(t)+\mathcal{F}(u(t))+ B w(t),
\end{equation}
where $A$ is an unbounded linear operator on a Hilbert space $H$, $\mathcal{F}$ is a nonlinear perturbation that can also be unbounded,
$B$ is a given control operator, 
$w$ is a control that can be chosen and takes values in $U$, a given Hilbert space (possibly finite-dimensional), and $u(t) \in H$ is the state of the system.
An interesting question in control theory is to know whether we can rapidly stabilize the system with the control $w$, which can be formulated as follows
\begin{problem}
For any $\lambda>0$, does there exist an operator $K=K_{\lambda}\in\mathcal{L}(D(A);U)$ such that, by choosing $w(t) = K_{\lambda}u(t)$, the system \eqref{eq:sys0} is (locally) exponentially stable with decay rate $\lambda$?
\end{problem}
\noindent ``Locally'' here only makes sense when $\mathcal{F}\neq 0$ and means locally around the equilibrium $u^{*}=0$. 

This question has been extensively investigated in the last decades in different frameworks and under different assumptions on $A$, $B$ and $\mathcal{F}$. 
Even when the system is linear (i.e. $\mathcal{F}=0$), answering this question in all generality is challenging. 
The first works date back (at least) to Slemrod \cite{slemrod1972linear} in 1972 in the case where $\mathcal{F}=0$ and $B$ is a bounded operator. Many results in this framework were obtained 
using tools from optimal control and Linear-Quadratic (LQ) theory by Lions, Barbu, Lasiecka, Triggiani and many others \cite{lions1971optimal,barbu2018controllability,LT1,LT2,urquiza2005rapid,vest2013rapid}. This question was considered in the semilinear framework, i.e. $\mathcal{F}(H)\subset H$ in \cite{trelat2017stabilization}. Other approaches successfully obtained results even when $B$ is unbounded; one can cite, for instance, the observability approach of \cite{trelat2019characterization,liu2022characterizations,ma2023feedback,kunisch2025frequency}, in particular \cite{trelat2019characterization} shows that when $(A,B)$ is exactly null controllable the system can be rapidly stabilized\footnote{See also \cite{zabczyk2020mathematical} for exponential stability without requesting rapid stabilization, i.e. arbitrary $\lambda$.} and \cite{liu2022characterizations} gives a very nice characterization of stabilizability in terms of observability. Other results, inspired by optimal control approaches and using Riccati or Hamilton-Jacobi-Bellman equations, were obtained on semilinear systems either on particular systems of interest (for instance, \cite{breiten2014riccati}, see also \cite{casas2017stabilization}) or when $A$ is parabolic (see Definition \ref{defn:par-op}) \cite{barbu2003feedback,barbu2018controllability}. One can also cite the Gramian approach (see, for instance, \cite{urquiza2005rapid}) and in particular the recent result of \cite{nguyen2024stabilization} where the author managed to modify the Gramian approach to get a quantitative rapid stabilization and even a finite-time stabilization for a semilinear system (i.e. $\mathcal{F}(H)\subset H$) when $A$ is skew-adjoint (see \cite{nguyen2024rapid} on the application to the 1D Schrödinger equation). Nevertheless, in these approaches, it often happens that the control feedback laws are not explicit, 
as they rely either on solving a minimization problem and an algebraic Riccati equation \cite{komornik1997rapid} (or even a Hamilton-Jacobi-Bellman equation
\cite{bellman1952theory}, \textcolor{black}{see also \cite{kunisch2023learning} for a learning approach to alleviate this difficulty}) or because they rely on the knowledge of the semigroup $e^{A^{*}t}$.\\
%

Other methods have been introduced, specifically for the parabolic cases. For instance, the impressive Frequency Lyapunov method which allows to obtain a quantitative rapid stabilization and even a finite-time stabilization for the multidimensional heat equation \cite{xiang2024quantitative} (and later for the 2D Navier-Stokes equation \cite{xiang2023small}) by relying on Carleman estimates and a specific Lyapunov function. However, it requires the control to be distributed, that is $K$ takes values in $U=H$. For this most challenging case, where $B$ is unbounded and the control belongs to a finite-dimensional space (that is $K$ takes values in $U=\mathbb{R}^{k}$ for some $k\in\mathbb{N}\setminus\{0\}$) it is worth highlighting, still in the parabolic framework, the work of \cite{badra2014fattorini} where the authors managed to deal both with a wide class of linear parabolic systems and, notably, with some nonlinear systems where the nonlinear perturbation $\mathcal{F}$ is not semilinear (i.e. $\mathcal{F}(H)$ is not contained in $H$) as long as the system is approximately controllable \cite{badra2011stabilization}.\\


Another method was introduced to tackle the aforementioned limitations and deal with this most challenging case in a general framework: the \emph{$F$-equivalence} (for \emph{feedback equivalence}). \textcolor{black}{This name was in fact first introduced by Brunovsky in \cite{brunovsky1970classification} for linear finite-dimensional system.}
 The principle is simple: instead of trying directly to find a feedback $K$, this method solves a different mathematical problem:
\begin{problem}
\label{pb:2}
Given operators $(A, B)$ and a target operator $\widetilde{A}$, find $(T,K)\in \mathcal{L}(H)\times \mathcal{L}(D(A);U)$ such that $T$ is an isomorphism from $H$ into itself and maps (in $H$) the system
\begin{equation}
\label{eq:introor}
\partial_{t}u(t) =Au(t)+BKu(t)
\end{equation}
to the system
\begin{equation}
\label{eq:introtar}
\partial_{t}u(t) =\widetilde{A}u(t).
\end{equation}
\end{problem}

Of course, if $\widetilde{A}$ generates an exponentially stable semigroup on $H$, the existence of such a pair implies the exponential stability of the original system \eqref{eq:introor} in $H$ (see Proposition \ref{prop:sg}). This approach is sometimes called \emph{generalized backstepping} or \emph{Fredholm backstepping} for, when $T$ is a Volterra transform of the second kind, it coincides with the well investigated \emph{backstepping} method for 1D systems \cite{BaloghKrstic,KrsticBook,hu2015control} (see also \cite{xiang2019null} or \cite{cauvin2023boundary} and references therein).
 In the last ten years, $F$-equivalence approaches have been used to achieve rapid stabilization of many systems, first for particular systems \cite{Coron2014-local,CoronLu15,coron2017finite,Coron2018-rapid,deutscher2019fredholm,gagnon2021fredholm,xiang2022fredholm,Coron2022-stabilization,lissy2023rapid} and recently in increasingly general settings \cite{xiang2024fredholm,hayat2024fredholm}.\\

By definition, the $F$-equivalence problem is \emph{a priori} asking for more than the rapid stabilization, which is only its consequence. 
 However, it actually turns out that for many systems the sufficient conditions of existence of an $F$-equivalence are relatively permissive and for skew-adjoint systems they were shown to be even better than the usual known sufficient conditions for rapid stabilization \cite[Section 3.1]{hayat2024fredholm}. 
\textcolor{black}{This can be explained since the $F$-equivalence allows us to look at the problem directly as a stabilization problem rather than deriving a feedback from the resolution of the (optimal) control problem. This avoids usual admissibility conditions on $B$ (see, for instance, \cite[Section 2.3]{coron2007control} or \cite{tucsnak2009observation}) which usually ensures that the system \eqref{eq:sys0} is well-posed for 
a whole class of controls,
which is not needed for the stabilization problem (one only needs the system to be well-posed along $w = Ku$). By considering the equivalence with the simpler system \eqref{eq:introtar}, it also avoids additional regularity conditions on $K$ as long as \eqref{eq:sys0} is well-posed. This, together with the explicitness of the feedback constructed makes the $F$-equivalence interesting both as a problem and for its application to the rapid stabilization.
} 
 
However, when it comes to parabolic systems, the existing $F$-equivalence conditions are likely too conservative compared to usual condition of rapid stabilization: they are stronger than asking for the exact null controllability of $(A,B)$ \textcolor{black}{as in \cite{trelat2019characterization,Datko1971}, and hence stronger than approximate controllability as in \cite{badra2014fattorini}.}

%
Another limit is that all the existing results of $F$-equivalence assume uniformly bounded multiplicities of the eigenvalues of $A$. For skew-adjoint systems this condition is necessary as soon as there is a finite number of controls (i.e. $K$ takes values in $U=\mathbb{R}^{k}$). For parabolic systems, however, this is likely too conservative as well. While, strictly speaking, this does not restrict this approach to 1D systems, it is still a strong limitation in practice when looking at systems that are multidimensional in space.\\

In this paper, 
we tackle these two limitations, hence from now on we have $U=\R^k$ with some $k\in \N \setminus \{0\}$. We show the following result (see Theorem \ref{thm:main} for a more detailed version):
\begin{thm}
\label{th:th01}
Let $A$ be a parabolic (unbounded) operator on a Hilbert space $H$ with a Riesz basis of eigenvectors (see Definition \ref{defn:par-op}) and $B\in (\textcolor{black}{D(A^{*})'})^{k}$. For all $\lambda\in \R_{>0}$, there exists an explicitly computable $m(\lambda)$ (depending only on $A$) such that either
\begin{itemize}
\item $k< m(\lambda)$ and there is no exponentially stable operator $\widetilde{A}$ such that there exists a solution to the $F$-equivalence problem \ref{pb:2}.
\item $k \geq m(\lambda)$, then if $B$ satisfies the $\lambda$-approximate controllability condition \hyperref[hyp:B]{$(H_{B})$},  
there exists an explicit $\widetilde{A}$, $T$ and $K\in\mathcal{L}(H;\mathbb{R}^{k})$ such that $\widetilde{A}$ is exponentially stable with decay rate $\lambda$ and $(T,K)$ are solutions of the $F$-equivalence problem \ref{pb:2}.
\end{itemize}
\end{thm}
This implies in particular that the original system \eqref{eq:sys0} with $w(t)=Ku$ and $\mathcal{F}=0$ is well-posed (see Proposition \textcolor{black}{\ref{prop:sg}}). The assumption \hyperref[hyp:B]{$(H_{B})$} is central but also very natural for such a parabolic system. In particular it is weaker than the approximate controllability of $(A,B)$ and weaker than the usual conditions required for the $F$-equivalence (see the discussion below Theorem \ref{th:th02}).\\

A consequence of this theorem is the exponential stability of the nonlinear system \eqref{eq:sys0}.
\textcolor{black}{
From now on we will work with a diagonal parabolic operator $A$ on $H$ (see Definition \ref{defn:par-op}). Let $s\in \R$, we set $\|x\|_{D_{s}(A)} = \|(-A+\delta)^s x\|_H$, and $D_{s}(A)$ is the Hilbert space associated to this norm, completed if $s<0$ (with $\delta\geq 0$ such that $-A+\delta$ is invertible). Here we state the assumption that the nonlinearity $\mathcal{F}$ must satisfy.
\begin{hyp}
\label{item:F1} We set $\gamma = \min(1-s,1/2)$ where $s$ is such that $B\in (D_{-s}(A))^k$. The mapping $\mathcal{F}$ goes from $D_{\gamma}(A)$ to $D_{-1/2}(A)$, such that there exists $\eta , K>0$ and $\Phi : \R_{\geq 0} \rightarrow \R_{\geq 0}$ non-decreasing continuous at 0 with $\Phi(0)=0$, which satisfies the following conditions: for all $u,v\in D_{\gamma}(A)$ with $\|u\|_{H}, \|v\|_H \leq \eta$, we have
\begin{align}
\label{eq:FH1}
\|\mathcal{F}(u)\|_{D_{-1/2}(A)}
    &\le \Phi(\|u\|_{H}) \,\|u\|_{D_{\gamma}(A)}, \\
\label{eq:FH2}
\|\mathcal{F}(u) - \mathcal{F}(v)\|_{D_{-1/2}(A)}
    &\le K(\|u\|_{D_{\gamma}(A)} + \|v\|_{D_{\gamma}(A)})\,\|u-v\|_{H} \\
    &\quad + K\Phi(\|u\|_{H} + \|v\|_{H})\,\|u-v\|_{D_{\gamma}(A)} \notag.
\end{align}
\end{hyp}
One can look at Remarks \ref{rem:rmkF1}--\ref{rem:NL-sufficient} for a discussion on this assumption.
With this assumption we can prove the following exponential stability result (see Theorem \ref{thm:nonlinear-stab} and the associated remarks for further details).}

\begin{thm}
\label{th:th02}
Let $A$ be a parabolic (unbounded) operator on a Hilbert space $H$ with a Riesz basis of eigenvectors and let $B\in (D_{-s}(A))^{k}$ (with $s \in [0,1]$). For any $\lambda>0$, if 
\begin{itemize}
\item $B$ \textcolor{black}{satisfies the $\lambda$-approximate controllability condition \hyperref[hyp:B]{$(H_{B})$}}
 \item $\mathcal{F}$ satisfies Assumption \ref{item:F1}
 \end{itemize}
then there exists an explicit $K\in \mathcal{L}(H,\mathbb{R}^{k})$ such that the system \eqref{eq:sys0} is exponentially stable with $w(t)=Ku$.
\end{thm}

As intended, the $\lambda$-approximate controllability assumption \hyperref[hyp:B]{$(H_{B})$} is much less restrictive than the $F$-equivalence conditions in a general setting given by \cite{hayat2024fredholm}. \textcolor{black}{This either significantly improves or recovers the recent results of \cite{CoronLu15,gagnon2021fredholm, xiang2022fredholm,lissy2023rapid, hayat2024fredholm}.}
As an illustration, in the case of the $1d$ heat equation and Burgers' equation on a torus studied in \cite{xiang2022fredholm,hayat2024fredholm} with $B=(B_{1},B_{2})$ where $B_{1}: \;\textcolor{black}{x\mapsto}  \sum_{n \geq 1} b_n^1 \sin(n x)$ is odd and $B_{2}\;\textcolor{black}{x\mapsto} \sum_{n \geq 0} b_n^2 \cos(n x)$ is even, the $F$-equivalence condition of \cite{gagnon2021fredholm,hayat2024fredholm} amounts to 
\begin{equation}
b_0^2 \not = 0 \ \text{and} \  \exists \gamma \in [0,1/2), \ \forall j \in \{1,2\}, \ \forall n \geq 1, \  c \leq |b_n^j| \leq Cn^\gamma,
\end{equation} 
where $c$ and $C$ are positive constants and \textcolor{black}{our} $\lambda$-approximate controllability assumption \hyperref[hyp:B]{$(H_{B})$} amounts to the \textcolor{black}{(much)} less restrictive condition
\begin{equation}
\forall j \in \{1,2\}, \  \forall n \leq \sqrt{\lambda}, \ b_n^j \not = 0.
\end{equation}
In particular, if one wants to stabilize at any rate $\lambda$, the previous conditions amount to the approximate controllability of $(A,B)$ by the generalized Fattorini criterion (\textcolor{black}{see \cite{badra2014fattorini} and} Lemma \ref{lem:Fattorini}). 
The example of the Kuramoto-Sivashinsky system of \cite{CoronLu15} is discussed in Section \ref{subsec:KS-eq}. \\ 

In Theorem \ref{th:th02} there is also no requirement on the multiplicity of the eigenvalues of $A$, which makes it suitable, for instance, \textcolor{black}{for multidimensional systems in space, in contrast to essentially all the previous $F$-equivalence approaches \cite{Coron2014-local,CoronLu15,gagnon2021fredholm,lissy2023rapid,xiang2022fredholm,xiang2024fredholm,hayat2024fredholm}}.\\

While our main goal of this work is to improve the conditions for the $F$-equivalence problem for parabolic systems and extend them to multidimensional systems, one can note that the rapid stabilization result of Theorem \ref{th:th02} (see Theorem \ref{thm:nonlinear-stab} for a more detailed version) can still be compared to previous results for parabolic systems that use different approaches. In particular,
compared to \cite{raymond2019stabilizability}, the system is not necessarily linear and,
to \cite{badra2014fattorini,badra2020local}, the conditions on the nonlinearity $\mathcal{F}$ are similar although less restrictive on two points:
\textcolor{black}{
\begin{itemize}
    \item[--] First, the function $\Phi$ controlling the local Lipschitz behaviour of $\mathcal{F}$ is only required to be non-decreasing and continuous at zero with $\Phi(0)=0$, rather than locally Lipschitz as in \cite{badra2014fattorini}. This relaxation is essential for handling reaction terms $f \in C^{1,1}_{\mathrm{loc}}$ as in Subsection~\ref{subsec:quasilinear-heat}.
    \item[--] Second, and more fundamentally, our conditions on $\mathcal{F}$ are expressed directly with respect to the open-loop operator $A$, whereas the usual perturbative approaches express them with respect to the closed-loop operator $A+BK$ and therefore depend implicitly on the feedback $K$ being constructed. This is a genuine consequence of the $F$-equivalence: Lemma~\ref{lem:iso-scale} shows that the transformation $T$ is an isomorphism of $D_r(A)$ for every $r \in [-1, 1-s]$, which yields $D_\gamma(A+BK) = D_\gamma(A)$ with equivalent norms. We emphasise that this is not a cosmetic improvement: as already noted in \cite{badra2014fattorini}, the main difficulty to apply such a result in concrete examples is that one has to identify the spaces $H_{F,\alpha}$ for $\alpha \in [0,1]$, a difficulty that disappears in our framework. 
\end{itemize}}
\textcolor{black}{The two points are connected, the relaxation on $\Phi$ would be of limited practical value if the resulting condition still had to be checked on the implicitly-defined space $D_\gamma(A+BK)$.}
Note that compared to \cite{badra2014fattorini}, and similarly to \cite{badra2011stabilization}, $B$ can belong to $(D_{-1}(A))^{k}$ and does not have to belong to the smaller space $D_{-s}(A)^{k}$ for some $s\in [0,1)$. 

In Section \ref{subsec:app} we illustrate this result on several examples of applications. Among others, \textcolor{black}{we study the} rapid stabilization of \textcolor{black}{a heat equation on a Riemannian manifold,} the (nonlinear) Kuramoto-Sivashinsky equation, the Navier-Stokes equations and a quasilinear heat equation. 

Overall, the paper is organized as follows: in Section \ref{sec:setting} we introduce our setting, notations and useful propositions. In Section \ref{sec:main} we state our main results, i.e. Theorems \ref{thm:main} and \ref{thm:nonlinear-stab} that are shown in Section \ref{sec:MainProof}, and in Section \ref{subsec:app} we give examples of applications on concrete systems.
}

While revising this manuscript we were made aware of another preprint \cite{gagnon2025fredholm} treating the $F$-equivalence problem for parabolic systems in the particular case of linear systems with a similar method but relying on Cauchy matrices to derive the control. It would be an interesting question to know whether such an approach could also be applied to the general systems we consider in this paper.

\section{Settings and notations}
\label{sec:setting}
\subsection{Functional setting}

Let $(H, \scalp{\cdot}{\cdot}_H )$ be a Hilbert space \textcolor{black}{and $A$ be an unbounded operator on $H$ satisfying the following conditions:}
\begin{enumerate}[label=(A\arabic*)]
    \item \label{item:A1} There exists a family of eigenvectors $(e_{n})_{n \geq 1}$ of $A$ that is a Riesz basis of $H$.
    Then for every $x \in H$, $Ax$ has a meaning \textcolor{black}{(not necessarily in $H$)} and we define the domain of $A$ as \[D(A)= \left\{ x \in H \ | \ Ax \in H \right\}.\]  
    \item \label{item:A2} The sequence $(\Re(\lambda_n))_{n \geq 1}$, \textcolor{black}{where $\lambda_{n}$ are the eigenvalues associated to $(e_{n})_{n\in\mathbb{N}^{*}}$}, is non-increasing and we have \[ \Re(\lambda_n) \underset{n \rightarrow + \infty}{\rightarrow} - \infty .\]
    \item \label{item:A3} There exists  $C > 0$ such that
    \begin{equation*}
        \forall n \geq 1, \ |\Re(\lambda_n)| \geq C |\Im(\lambda_n)|.
    \end{equation*}
\end{enumerate}

    \begin{remark}
In the following, and for simplicity, we replace the Riesz basis assumption in \ref{item:A1} by the assumption that $(e_n)_{n\geq 1}$ is an orthonormal Hilbert basis. This does not affect our results, since any Riesz basis becomes orthonormal under an equivalent inner product, the only distinction is discussed in Remark \ref{remark:RieszHilbert}.
\end{remark}

\begin{defn}\label{defn:par-op}
    Let $A$ be an unbounded operator on $H$. We will say that $A$ is diagonal parabolic if it satisfies \ref{item:A1}, \ref{item:A2} and \ref{item:A3}.
\end{defn}

In particular, any parabolic self-adjoint operator is a diagonal parabolic operator.

\begin{remark}
Hypothesis \ref{item:A2} implies that the multiplicity of each eigenvalue of $A$ is finite, but the supremum of all such multiplicities can be infinite. Hypothesis \ref{item:A2} is here to ensure that $A$ is the infinitesimal generator of an analytic semigroup.
\end{remark}

For $A$  diagonal parabolic operator, hypothesis \ref{item:A2} ensures that $\{ \Re(\lambda_n) \ | \ n \geq 1 \}$ has a maximum $m_A$. We define $c_A := \max(0, m_A)$, this constant will be useful for stating our main result. \textcolor{black}{Besides, from \ref{item:A1} there exists an inner product on $H$ such that $(e_{n})_{n \geq 1}$ is orthonormal and the associated norm is equivalent to the norm associated with $\scalp{\cdot}{\cdot}_H$ (see \cite[Section 2]{hayat2024fredholm}). Therefore in the following we will assume, without loss of generality, that $e_{n}$ is an orthonormal basis of $H$. Also, for every $x \in H$ we define $x_n$ to be $\scalp{x}{e_n}_H$}. \\


\textcolor{black}{Since $A$ is} closed, $D(A)$ is a Hilbert space endowed with the inner product \begin{equation}
    \forall y, z \in D(A), \ \scalp{y}{z}_{D(A)} := \scalp{y}{z}_H + \scalp{Ay}{Az}_H.
\end{equation}
Notice that \textcolor{black}{\((e_n/\sqrt{1 + |\lambda_n|^2})_{n \geq 1}\)} forms an orthonormal basis of \(D(A)\). We see that \((D(A), H, D(A)')\) is a Gelfand triple.

Finally, note that a parabolic diagonal operator \(A\) is normal, in particular we have $D(A)=D(A^*)$. This allows us to see it as an element of \(\mathcal{L}(H, D(A)')\).

Later, \textcolor{black}{it will be useful to consider} \(AM\) where \(M \in \mathcal{L}(D(A)')\), and we wish that \(AM \in \mathcal{L}(H, D(A)')\). However, this is not generally the case, so we need to work with another operator algebra. This is the reason for the following proposition.

\begin{prop}\label{prop:algebra}
     We define the following algebra
     \begin{equation}
        \mathcal{L}_H(D(A)') = \{ M \in \mathcal{L}(D(A)') \ | \ M_{|H} \in \mathcal{L}(H)\},
     \end{equation}
     \textcolor{black}{which,} endowed with the norm,
     \begin{equation}
         ||M||_{\mathcal{L}_H(D(A)')} := \underset{\substack{||x||^2_{D(A)'} + ||y||_H^2 = 1 \\ (x, y) \in D(A)' \times H}}{\textnormal{sup}} \sqrt{||Mx||^2_{D(A)'} + ||My||_H^2},
     \end{equation}
     \textcolor{black}{is} a Banach algebra. Furthermore, we have the embedding of \(\mathcal{L}_H(D(A)')\) in \(\mathcal{L}(D(A)' \times H)\)
     \begin{equation}
         \forall M \in \mathcal{L}_H(D(A)'), \ \forall (x, y) \in D(A)' \times H, \ \varphi(M)(x, y) := (Mx, My). 
     \end{equation}
\end{prop}

\begin{proof}
    It is clear that \(\mathcal{L}_H(D(A)')\) is an algebra and that \(\varphi\) is injective. Notice that the norm given above was designed such that 
    \[
    \forall M \in \mathcal{L}_H(D(A)'), \ ||\varphi(M)||_{\mathcal{L}(D(A)' \times H)} = ||M||_{\mathcal{L}_H(D(A)') }.
    \]
    Hence, \(\mathcal{L}_H(D(A)')\) is a Banach algebra.
\end{proof}

If we denote by \(\mathcal{GL}_H(D(A)')\) the group of invertible elements in \(\mathcal{L}_H(D(A)')\), then the above embedding gives us the following characterization:
\begin{equation}
    M \in \mathcal{GL}_H(D(A)') \iff (M, M_{|H}) \in \mathcal{GL}(D(A)') \times \mathcal{GL}(H).
\end{equation}
Notice that now, for every \(M \in \mathcal{L}_H(D(A)')\), we have \(AM \in \mathcal{L}(H, D(A)')\).

\subsection{Generalized Sobolev spaces}\label{subsec:sobolev}
To define the generalized Sobolev spaces, let us fix \(\delta \geq 0\) such that \(-A + \delta\) is invertible.\footnote{Note that it is sectorial since $A$ is a diagonal parabolic operator.} We denote here \textcolor{black}{and in the following} $-A+\delta$ for \( -A + \delta I \). Then we can define 
\begin{equation}\label{eq:space-scale-A}
    \forall s \in \mathbb{R}, \quad D_s(A) := D((-A + \delta)^s).
\end{equation}
Endowed with the usual graph norm, and after metric completion if $s<0$, these become Hilbert spaces. Let \(s \in \mathbb{R}\), one can show that the following norm is equivalent
\begin{equation}
    \forall x = \sum_{n \geq 1} x_n e_n \in D_s(A), \quad \|x\|_{D_s(A)}^2 := \sum_{n \geq 1} (1 + |\lambda_n|^2)^s |x_n|^2.
\end{equation}
We can identify \(D_{-s}(A)\) with \(D_s(A)^\prime\), and the triple \((D_{-s}(A), H, D_s(A))\) forms a Gelfand triple. Note that \(D_1(A) = D(A)\) \textcolor{black}{and $D_{0}(A)=H$}. 

We refer to this scale of spaces as generalized Sobolev spaces, as they share the same properties as classical Sobolev spaces. For instance, when \(A\) is a power of the Laplace operator on a closed manifold, these spaces coincide with the usual Sobolev spaces. For more details, see Subsection \ref{subsec:manifolds-settings} \textcolor{black}{and in particular \eqref{eq:scale-comparaison}}.

\subsection{Frequency decomposition} Let $\lambda > 0$. Here, we focus on defining a decomposition of our spaces into \textcolor{black}{l}ow and high frequencies. We define the low and high frequency spaces as
\begin{equation}
\label{eq:defLlambda}
\begin{aligned}
    L_\lambda &= \text{span}\left\{e_n \ \middle| \ \Re(\lambda_n) \geq -\lambda \right\}, \\
    H_\lambda &= \overline{\text{span} \left\{e_n \ \middle| \ \Re(\lambda_n) < -\lambda \right\}}^{H}, \\
    D(A)'_\lambda &= \overline{\text{span} \left\{e_n \ \middle| \ \Re(\lambda_n) < -\lambda \right\}}^{D(A)'}.
\end{aligned}
\end{equation}
We have the orthogonal decomposition \(H = L_\lambda \oplus H_\lambda\) and \(D(A)' = L_\lambda \oplus D(A)'_\lambda\). We set \(N(\lambda) := \dim L_\lambda < +\infty\) and define \(m(\lambda)\) to be the greatest multiplicity of any eigenvalue in \(L_\lambda\). \textcolor{black}{We denote by \(P_{L}\) and \(P_{H}\) the orthogonal projections on \(L_\lambda\) and \(H_\lambda\) in \(H\).}

\subsection{Control setting}
In this paper, we seek to stabilize $A$ using only a finite number of scalar controls, which means that our control system looks like
\[ \partial_t u = Au + Bw(t), \]
with $w(t) \in \C^m$ and $B$ a given control operator. Let $E$ be a normed vector space. We will use the canonical isomorphism from $E^m$ to $\mathcal{L}(\C^m, E)$ to identify an element $B \in \mathcal{L}(\C^m, E)$ with $(B_1, \dots, B_m) \in E^m$ in the following way
\begin{equation}\label{eq:canonical-iso}
    B : z \in \C^m \mapsto \sum_{j=1}^m z_j B_j \in E.
\end{equation}



\textcolor{black}{Let \(B \in (D(A)')^m\) for some fixed \(m \geq 1\),} when the supremum of the multiplicities of \(A\) is infinite, the pair \((A, B)\) is not approximately controllable, making it impossible to achieve stabilization at any desired rate. Instead, we set a target stabilization rate \(\lambda > 0\) and seek to determine whether we can stabilize the system at this rate. \textcolor{black}{Note that, for any given $\lambda>0$, we have the following lemma thanks to Appendix G of \cite{xiang2024fredholm}.}

\begin{lem}
\label{lem:decomp1}
There exist $N_1(\lambda), \dots, N_{\ml}(\lambda) \in \N^*$ and a partition $(e_n^1)_{n \geq 1}, \dots, (e_n^{\ml})_{n \geq 1}$ of $(e_n)_{n \geq 1}$ such that, if we set $L_\lambda^j = \text{span}((e_n^j)_{1 \leq n \leq N_j(\lambda)})$ \textcolor{black}{and  \(\mathcal{H}^j = \overline{\text{span}((e_n^j)_{n \geq 1})}^{H}\) (thus \(H = \bigoplus_{j=1}^{m(\lambda)} \mathcal{H}^j\)}) , then the multiplicities of eigenvalues are simple in $L_\lambda^j$ and
\[ L_\lambda = \bigoplus_{k=1}^{\ml} L_\lambda^k,\;\;\textcolor{black}{H = \bigoplus_{j=1}^{m(\lambda)} \mathcal{H}^j}.\]
\textcolor{black}{Moreover,} for each \(j \in \{1, \dots, \textcolor{black}{m(\lambda)}\}\), \(A\) induces a diagonal parabolic operator on \(\mathcal{H}^j\) such that 
\[ A = A_1 + \dots + A_{\textcolor{black}{m(\lambda)}}, \quad D(A) = \bigoplus_{j=1}^{\textcolor{black}{m(\lambda)}} D(A_j), \quad D(A)' = \bigoplus_{j=1}^{\textcolor{black}{m(\lambda)}} D(A_j)'.\]
Here $D(A_j)'$ is the dual of $D(A_j)$ with respect to the pivot $\mathcal{H}^j$, hence $D(A_j)'$ can be seen as a subspace of $D(A)'$.
\end{lem}

\textcolor{black}{In the following, w}e define $P_L^j$ as  the orthogonal projection onto $L_\lambda^j$ in $\H^j$. We then set $P_H^j = Id_{\H^j} - P_L^j$. 
\textcolor{black}{In order to achieve a stabilization at decay rate $\lambda$, we make the following assumption on our control operator}:\\

\begin{enumerate}

\item[($H_{B}$)]
\label{hyp:B} 
$\textcolor{black}{m\geq} \ml$, \textcolor{black}{and} for all $ j \in \{1, \dots, \ml\}$, we have $B_j \in D(A_j)'$ and  
\begin{equation*}
  \scalp{B_j}{e_n^j}_{D(A)',D(A)} \not = 0 ,\;\; \textcolor{black}{\forall n \in \{1, \dots, N_j(\lambda) \}.} 
 \end{equation*}
\end{enumerate}

\begin{remark}\label{remark:RieszHilbert}
In the case where $(e_n)_{n\geq 1}$ is a Riesz basis, then $B \in D(A^*)'$, and if $(\widetilde{e_n})_{n\geq 1}$ denote the biorthogonal family of $(e_n)_{n\geq 1}$, condition \hyperref[hyp:B]{$(H_{B})$} reads as follows: for all $j\in\{1,\dots,\ml\}$, we have $B_j\in D(A_j^*)'$ and
    \begin{equation*}
        \scalp{B_j}{\widetilde{e_n}^j}_{D(A^*)',D(A^*)} \not = 0 ,\;\; \textcolor{black}{\forall n \in \{1, \dots, N_j(\lambda) \}.} 
    \end{equation*}
Notice that if $(e_n)_{n\geq 1}$ is a Hilbert basis, this is exactly what was stated before.
\end{remark}

\begin{defn}
\label{def:Fadmissible}
Let $B = (B_1, \dots, B_{m}) \in (D(A)')^{m}$. We say that $B$ is $F_\lambda$-admissible if it satisfies \hyperref[hyp:B]{$(H_{B})$}.

\end{defn}

Let us briefly comment on \textcolor{black}{this assumption $(H_{B})$.
W}e first allow our control operators to be unbounded, which means they belong to a larger space than \(H\). For well-posedness reasons, we know that \(D(A)'\) is optimal, 
and here it is allowed, 
which makes it slightly \textcolor{black}{less restrictive} than \textcolor{black}{the condition of} \cite{badra2014fattorini} \textcolor{black}{where $B\in D_{-s}(A)^{m}$ with $s <1$} \textcolor{black}{and similar to the condition of \cite{badra2020local} \textcolor{black}{(which consider in addition a non-autonomous setting)}}. \textcolor{black}{Then the condition on the scalar product of the $B_{j}$ in $(H_{B})$} is here to ensure that the low-frequency system is controllable.\\

Finally, we define the concept of target operator, which we will need to define the concept of $F$-equivalence.

\begin{defn}
\label{def:lambda-target}
Let \(A\) be a diagonal parabolic operator and \(\lambda > 0\). We say that an unbounded normal operator \(D\) on \(H\) is a \(\lambda\)-target if it is the infinitesimal generator of a differentiable semigroup on \(H\) with a growth rate of at most \(-\lambda\). This means that
\[ 
\exists C > 0, \forall x \in H, \forall t \geq 0, \ ||e^{tD} x||_H \leq C e^{-\lambda t} ||x||_H. 
\]
\end{defn}
\begin{remark}
\textcolor{black}{In Definition \ref{def:lambda-target}, the operator $D$ could have a domain different from $A$, and this happens, for instance, in \cite{Coron2022-stabilization}. In the following, however,} we will only consider $\lambda$-targets having the same domain as A.    
\end{remark}
\section{Main results}
\label{sec:main}
\subsection{$F$-equivalence results}\label{subsec:results}

Let \(\lambda > 0\), \(A\) be a diagonal parabolic operator in \(H\), \(D\) be a \(\lambda\)-target \textcolor{black}{with domain $D(A)$}, and \(B \in (D(A)')^{m}\).
We can now define the concept of \(F\)-equivalence.

\begin{defn}[$F$-equivalence]\label{defn:df-eq}
Let \((T, K) \in \mathcal{GL}_H(D(A)') \times \mathcal{L}(H, \C^{m})\). We say that \((T, K)\) is an \(F\)-equivalence of \((A, B, D)\), or that it is an \(F\)-equivalence between \((A, B)\) and \(D\), if 
\begin{equation}\label{eq:f-eq}
    \begin{cases}
        T(A + BK) = D T \ \text{in} \ \mathcal{L}(H, D(A)'), \\
        TB = B \ \text{in} \ D(A)'.
    \end{cases}
\end{equation}
Furthermore, if \(K \in \mathcal{L}(L_\lambda, \C^{m})\), we say that \((T, K)\) is a parabolic \(F\)-equivalence.
\end{defn}

Before presenting our main result, we want to emphasize a few points. As one can imagine, finding an $F$-equivalence is a challenging problem. The condition $TB = B$ in \eqref{eq:f-eq} is included for two reasons. First, to make the problem linear in $(T, K)$, and secondly, in the hope of achieving uniqueness, i.e., that there exists one and only one $F$-equivalence of $(A, B, D)$, which greatly aids in finding a solution. As Theorem \ref{thm:ffeq} shows \textcolor{black}{(see also \cite{Coron2015})}, if $(A, B)$ is finite-dimensional and with $D = A - \lambda$, then there exists one and only one $F$-equivalence of $(A, B, D)$. Unfortunately, in our case, we will show in Section \ref{section:wfeq} that in general, there is no uniqueness \textcolor{black}{to the $F$-equivalence problem}. However, at the same time, we introduce a new formalism that we call \emph{weak $F$-equivalence}, which allows us to regain uniqueness\textcolor{black}{. More} precisely, we show that the uniqueness is linked with the approximate controllability of $(A,B)$. For more details, see Section \ref{section:wfeq}. \textcolor{black}{Our first main theorem is the following}

\begin{thm}[Parabolic \(F\)-equivalence]\label{thm:main}
Let \(A\) be a diagonal parabolic operator, and let \(\lambda \in \R_{>0} \). Suppose \(B \in (D(A)')^{m(\lambda)}\) is an \(F_\lambda\)-admissible control operator \textcolor{black}{(see Definition \ref{def:Fadmissible})}. For \(\mu \geq \lambda + c_A\), we define
\textcolor{black}{\begin{equation}
    D = (A_L - \mu) P_L + A_H P_H.
\end{equation}}
\textcolor{black}{Then, $D$ is a \(\lambda\)-target and}  for almost every \(\mu \geq \lambda + c_A\), there exists a parabolic \(F\)-equivalence \((T, K)\) between \((A, B)\) and \(D\).

\end{thm}

{\color{black}
\begin{remark}
Note that the choice of $\lambda$-target $D$ does not depend on $B$ and only depends on $\lambda$, $\mu$ and $A$.
\end{remark}
}
The proof \textcolor{black}{of Theorem \ref{thm:main}} is provided in Subsection \ref{subsec:feq-mainproof}. Let us comment on the above theorem. First, the definition of $D$ is very natural \textcolor{black}{if the goal is to obtain a $\lambda$-target, that is an operator with growth rate at most $-\lambda$. T}he hope behind \textcolor{black}{is} to find \textcolor{black}{a feedback operator $K$} acting only on the low-frequency space, as \textcolor{black}{one can expect} intuitively and as \textcolor{black}{is classically used for parabolic systems 
see, for instance, \cite{badra2014fattorini,trelat2017stabilization,xiang2024quantitative}}. Secondly, besides $D$ being intuitive, it is novel to perform $F$-equivalence \textcolor{black}{in a generic framework with a target operator different from\footnote{\textcolor{black}{Note that another target operator was also used in \cite{Coron2022-stabilization} in the particular case of the Saint-Venant system.}}} $A - \lambda$, as done in \cite{xiang2022fredholm, xiang2024fredholm, hayat2024fredholm}. \textcolor{black}{This is, of course, made possible by the parabolic nature of the system.}  

In fact, here, it would have been impossible to take $D = A - \lambda$, because if, for example, $B \in H$, \textcolor{black}{then} $BK$ would be a compact operator on $H$. If there exists $T \in \mathcal{GL}(H)$ such that
\[ T(A + BK) = (A - \lambda)T, \]
we should have $\sigma(A + BK) = \sigma(A - \lambda)$. However, we know that \textcolor{black}{if} $BK$ is compact, $A$ and $A + BK$ should asymptotically have the same spectrum (see, for instance, \cite[Chapter IV, Sec. 1]{engel1999one}), which would be absurd.

\subsection{\textcolor{black}{Rapid s}tabilization results} Here, we apply $F$-equivalence to the stabilization of parabolic systems. We start with the linear case, then we consider nonlinear equations. As before, let \(\lambda > 0\), \(A\) be a diagonal parabolic operator in \(H\), \(D\) be a \(\lambda\)-target, and \(B \in (D(A)')^{m}\). \textcolor{black}{As expected, finding a solution to the $F$-equivalence problem also ensures rapid stabilization of the linear system:}
\begin{prop}\label{prop:sg}
    Let $(T, K) \in \mathcal{GL}_H(D(A)') \times \mathcal{L}(H, \C^{m})$ be an $F$-equivalence of $(A, B, D)$. Then $A+BK$ is an unbounded operator on $H$ with dense domain $T^{-1}(D(A))$ which generates \textcolor{black}{a differentiable} semigroup with a growth of at most $-\lambda$. \textcolor{black}{In particular,} the Cauchy problem 
\begin{equation}
    \begin{cases}
        \partial_t u = (A + BK)u, & \forall t > 0, \\
        u(0) = u_0 \in H,
    \end{cases}
\end{equation}
is well-\textcolor{black}{posed 
in}
$C^0([0, +\infty); H) \cap C^{\infty}(0, +\infty; H)$ with $u(t) \in D(A + BK), \, \forall t > 0$, \textcolor{black}{and we have the following exponential} stability estimate:
\begin{equation}
    \exists C > 0, \forall u_0 \in H, \forall t \geq 0, \ ||u(t)||_H \leq C e^{-\lambda t} ||u_0||_H.\footnote{Here $u(\cdot)$ is the solution with initial condition $u_0$.}
\end{equation}
\end{prop}
\begin{proof}
    \textcolor{black}{See Appendix \ref{app:proofstab}.}
\end{proof}

The above proposition demonstrates the utility of the $F$-equivalence approach for the stabilization of linear systems. To show that $(A, B)$ is exponentially stable at rate $\lambda$, one only needs to demonstrate the existence of an $F$-equivalence between $(A, B)$ and some target operator $D$ as described above. Additionally, notice that this $F$-equivalence approach ensures that the problem is well-posed.\\

Now, using $F$-equivalence we aim to stabilize the following type of nonlinear control system \begin{equation}
    \d_t u = Au + Bw(t) +  \mathcal{F}(u),
\end{equation} 
\textcolor{black}{w}here $\mathcal{F}$ is a nonlinear map, possibly highly nonlinear (i.e. not semilinear) and with regularity as low as that of the operator $A$ \textcolor{black}{(see Assumption \ref{item:F1})}.

More precisely, let $(T,K)$ be an $F$-equivalence of $(A,B,D)$ given by Theorem \ref{thm:main}, we aim to show local well-posedness and exponential stability of the following system
\begin{equation}
    \begin{cases}
        \partial_t u = (A + BK)u + \mathcal{F}(u), \\
        u(0) = u_0 \in H,
    \end{cases}
\end{equation}
where $\mathcal{F}$ satisfies Assumption \ref{item:F1}.
As $D$ is a diagonal parabolic operator that generates an exponentially stable semigroup, $A+BK$ \textcolor{black}{is too, thanks to the $F$-equivalence property}. For this kind of operator, nonlinear perturbations are well-known, see \cite{pazy1983semigroups,trelat2017stabilization}: essentially, when we see $\mathcal{F}$ as a small Lipschitz perturbation, optimal conditions\footnote{Optimal in the sense that, as $\mathcal{F}$ goes from $D_{1/2}(A+BK)$ to $D_{-1/2}(A+BK)$, hence it is as regular as $A+BK$, allowing quasi-linear perturbations.} for well-posedness and stability are given in Appendix \ref{sec:Nonlin-ps}.

Such a direct perturbation approach would give us a result where the conditions on the unboundedness on $\mathcal{F}$ are stated with respect to $D_\gamma(A+BK)$, which is what typically happens in the existing literature, see \cite{badra2014fattorini,badra2011stabilization}. Here the $F$-equivalence will allow us to give unboundedness conditions with respect to $D_\gamma(A)$ directly (see Assumption \ref{item:F1}). Indeed, thanks to Lemma \ref{lem:iso-scale}, if $B \in D_{-s}(A)$  with $s\in[0,1]$, then $T \in \mathcal{GL}(D_{1-s}(A))$ and this will ensure by $F$-equivalence that $D_r(A+BK) = D_r(A)$ for all $r \in [-1,1-s]$.
We begin with a few remarks about Assumption \ref{item:F1} before stating our main result.
\begin{remark}[\textcolor{black}{Admissibility of $B$ and domain of $\mathcal{F}$}]
\label{rem:rmkF1}
    Notice that, in Assumption \ref{item:F1}, if $B \in (D_{-1/2}(A))^m$ (for instance, if $B$ is an admissible control operator, which is very common in control problems) then $\gamma = \frac{1}{2}$. In this situation, we allow $\mathcal{F}$ to be defined on $D_{1/2}(A)$ and to take values in $D_{-1/2}(A)$, meaning that $\mathcal{F}$ may be \textcolor{black}{in some sense} as irregular as the operator $A$ itself. This is optimal in the sense that, for such generators, further lowering the regularity of $\mathcal{F}$ would in general destroy local well-posedness.
\end{remark}
\begin{remark}[Sufficient conditions on $\mathcal{F}$]\label{rem:NL-sufficient}
    In practice we often work with $\Phi=Id_\R$ (but not always, see Subsection \ref{subsec:quasilinear-heat}), then notice that a sufficient condition on $\mathcal{F}$ to ensure Assumption \ref{item:F1} is $\mathcal{F}(0)=0$ and \textcolor{black}{for any $u,v \in D_{\gamma}(A)$ with $\|u\|_{H}, \|v\|_H \leq \eta$}
     \begin{equation}\label{eq:NL-sufficient}
        \|\mathcal{F}(u)-\mathcal{F}(v)\|_{D_{-1/2}(A)} \leq C (\|u\|_{H} + \|v\|_{H}) \|u-v\|_{D_{\gamma}(A)}.
     \end{equation}
     Even with $\Phi=Id_\R$, \eqref{eq:FH1} itself is not sufficient in general to ensure the previous inequality. To understand where \eqref{eq:FH2} comes from, suppose that $\mathcal{F}(u)=\mathcal{B}(u,u)$ where $\mathcal{B}$ is a bilinear map satisfying, \textcolor{black}{for any $u,v \in D_{\gamma}(A)$ with $\|u\|_{H}, \|v\|_H \leq \eta$,}
     \begin{equation}\label{eq:FH1B}
        \|\mathcal{B}(u,v)\|_{D_{-1/2}(A)} \leq C  \|u\|_{H}\|v\|_{D_{\gamma}(A)}.
     \end{equation}
     Then, if $\mathcal{B}$ is symmetric  (as in Subsection \ref{subsec:KS-eq}), \eqref{eq:FH1} is indeed sufficient to ensure \eqref{eq:NL-sufficient}. Otherwise (as in Subsection \ref{subsec:quasilinear-heat} and \ref{subsec:Navier-Stokes}) from \eqref{eq:FH1B} we deduce that for all $u,v \in D_{\gamma}(A)$ \textcolor{black}{with $\|u\|_{H}, \|v\|_H \leq \eta$}
     \begin{equation}
\|\mathcal{B}(u,u) - \mathcal{B}(v,v)\|_{D_{-1/2}(A)}
    \le C(\|u\|_{H} \|u-v\|_{D_{\gamma}(A)} 
     + \|v\|_{D_{\gamma}(A)} \|u-v\|_{H}).
     \end{equation}
     Which ensures that Assumption \ref{item:F1} holds with $\Phi=Id_\R$ and $K=C$.
\end{remark}

The above discussion leads to the following theorem.
\begin{thm}\label{thm:nonlinear-stab}
    Let \(A\) be a diagonal parabolic operator, $B\in (D_{-s}(A))^m$ with $s \in [0,1]$ and  $\mathcal{F}$ be a map satisfying Assumption \ref{item:F1}.
Let \(\lambda \in \R_{>0}\), suppose that $B$ is an \(F_\lambda\)-admissible control operator, and let $(T,K)$ be an $F$-equivalence given by Theorem \ref{thm:main}. Then, there exists $\delta > 0$ such that for every $u_0 \in H$, if $\|u_0\|_H \leq \delta$, there exists a unique solution $u \in C^0_b([0,+\infty);H) \cap L^2((0,+\infty);D_{\gamma}(A))$ (with $\gamma=\min(1-s,\frac{1}{2})$) to the following system
\begin{equation}\label{eq:nonlinear-system}
        \begin{cases}
        \partial_t u = (A + BK)u + \mathcal{F}(u), \\
        u(0) = u_0.
        \end{cases}
    \end{equation}
    In addition, the system is exponentially stable, more precisely setting $C = ||T||_{\mathcal{L}(H)}||T^{-1}||_{\mathcal{L}(H)}$, we have 
    \begin{equation}\label{eq:stab-nl}
        \forall t \geq 0, \quad ||u(t)||_H \leq C e^{-{\lambda}t}||u_0||_H.
    \end{equation}
\end{thm}

\begin{remark}[\textcolor{black}{Comparison with \cite{badra2011stabilization,badra2014fattorini}}]\label{rem:NL-novelty}
\textcolor{black}{\textcolor{black}{The main goal of this paper is to study the $F$-equivalence of parabolic systems and the rapid stabilization result is only a useful consequence.} As mentioned earlier, stabilization \textcolor{black}{results with} nonlinear perturbations for parabolic systems is already known (see \cite{badra2011stabilization,badra2014fattorini}). \textcolor{black}{N}otice that our assumptions on $\mathcal{F}$ are similar. However, ours are slightly more general thanks to the use of $\Phi$,  
which allows us to have sharper hypotheses in some cases (see Subsection \ref{subsec:quasilinear-heat}). Apart from this, one contribution of $F$-equivalence here, is that it allows us to express the unboundedness condition
on $D_\gamma(A)$ and not on $D_\gamma(A+BK)$, which gives a simpler and more natural condition.}
\end{remark}

\begin{remark}[\textcolor{black}{Explicitness of $T$ and $K$}]
\textcolor{black}{
    Equation \eqref{eq:stab-nl} shows that for numerical applications, or for quantitative finite time stabilization as in \cite{xiang2024quantitative}, the knowledge of $T$ is crucial in order to compute $C$. Note that in the proof of Theorem \ref{thm:main}, we give a complete explicit expression of $T$, see Propositions \ref{prop:partial-uniqueness} and \ref{prop:dense-feq}.} 
\end{remark}

The proof of this result is done in Section \ref{subsec:stab-nonl-mainproof}, \textcolor{black}{and note that, it is in fact a direct corollary of Theorem \ref{thm:main} and results on well-posedness and stability in Appendix \ref{sec:Nonlin-ps}.} Moreover, as shown in the proof of Theorem \ref{thm:main}, $K$ is obtained by solving a finite-dimensional linear system of equations. Thus, the above theorem provides a very simple and explicit way to stabilize a whole class of nonlinear parabolic PDE.

\section{Applications and examples}\label{subsec:app}

\subsection{Heat equation on manifolds}\label{subsec:manifolds-settings} Let $(\M,g)$ be a compact oriented and connected $d$-dimensional Riemannian manifold, \textcolor{black}{and set} $H= L^2(\M)$, $A=\Delta_g$ and $(e_n)_{n \geq 1}$ an orthonormal basis of eigenvectors such that \begin{equation}
    0 = -\lambda_1 \leq -\lambda_2 \leq \dots \leq -\lambda_k \rightarrow +\infty.
\end{equation}
The Weyl law tells us that \(N(\lambda) \underset{\lambda \rightarrow +\infty}{\sim} \frac{\text{Vol}(\mathcal{M}) \omega_d}{(2\pi)^d} \lambda^{\frac{d}{2}}\), \textcolor{black}{where we recall that $N(\lambda) = \text{dim}(L_{\lambda})$ (see \eqref{eq:defLlambda}). Because the eigenvalues are non-decreasing, and from the definition of $L_{\lambda}$,} we have \(N(\lambda_n) = n\). Hence, we have
\begin{equation}
    |\lambda_n| \underset{n \rightarrow +\infty}{\sim} \frac{4\pi^2}{(\omega_d \text{Vol}(\mathcal{M}))^{\frac{2}{d}}} n^{\frac{2}{d}}.
\end{equation}
Using the classical characterization of Sobolev spaces on manifolds as in \cite{craioveanu2001old}, we have
\begin{equation}\label{eq:scale-comparaison}
    \forall s \in \mathbb{R}, \quad D_s(A) = H^{2s}(\mathcal{M}).
\end{equation}
Notice that $\Delta_g$ is a diagonal parabolic operator on $L^2(\M)$, so we have the following immediate corollary of our main result.
\begin{cor}
    Let $(\mathcal{M}, g)$ be a compact oriented and connected $d$-dimensional Riemannian manifold. We set $H = L^2(\mathcal{M})$, $A = \Delta_g$, and fix $\lambda \in \R_{>0}$. Let $B \in (H^{-2}(\M))^{\ml}$ be an $F_\lambda$-admissible control operator. Then there exists $K \in \mathcal{L}(L^2(\M), \R^{\ml})$ such that the Cauchy problem
    \begin{equation}
        \begin{cases}
            \d_t u = \Delta_g u + BKu, & \forall t > 0, \\
            u(0) = u_0 \in L^2(\mathcal{M}),
        \end{cases}
    \end{equation}
    is well-posed,\footnote{In the sense of Proposition \ref{prop:sg}.} and we have the following stability estimate:
    \begin{equation}
        \exists C > 0, \forall u_0 \in L^2(\mathcal{M}), \forall t \geq 0, \ ||u(t)||_{L^2} \leq C e^{-\lambda t} ||u_0||_{L^2}.
    \end{equation}
\end{cor}
Before presenting concrete examples, notice that we have all the information we need about the asymptotic behavior of $N(\lambda)$ and $\lambda_n$. Interestingly, we see that the growth rate of the eigenvalues is entirely governed by the topology of the manifold, while the prefactor depends on its geometry, specifically its volume.

Despite all this, we still have no information on $m(\lambda)$. In fact, it is well-known that $\Delta_g$ has ``generically'' simple eigenvalues, which implies $m(\lambda) = 1$ for all $\lambda > 0$. For a more precise definition, see \cite{micheletti1972perturbazione, uhlenbeck1976generic}. 

We conclude this subsection with some examples. Let \(\mathcal{M}\) be a compact oriented and connected Riemannian manifold of dimension $d \leq 3$. The following result combines two complementary regimes for stabilization with a Dirac control: (i) for an arbitrary manifold, slow decay rates always work; (ii) under the (generic) assumption of simple eigenvalues, any decay rate can be reached for almost every $p\in\mathcal{M}$.

\begin{prop}\label{prop:dirac-stab-manifolds}
Let $p\in\mathcal{M}$ and denote by $\delta_p$ the Dirac distribution at $p$.
\begin{itemize}
\item[(i)] For every $\lambda \in (0,-\lambda_{2})$, defining $K\in\mathcal{L}(L^{2}(\mathcal{M}),\R)$ by\footnote{Here, $d\mu_g$ is the measure induced by the Riemannian volume form.}
\begin{equation}
\forall f\in L^{2}(\mathcal{M}),\qquad Kf = -\lambda \int_{\mathcal{M}} f\,d\mu_g,
\end{equation}
the Cauchy problem
\begin{equation}
\begin{cases}
\partial_t u = \Delta_g u + \delta_p\, Ku, & \forall t>0,\\
u(0)=u_0\in L^{2}(\mathcal{M}),
\end{cases}
\end{equation}
is well-posed and there exists $C>0$ such that
$\|u(t)\|_{L^{2}}\leq Ce^{-\lambda t}\|u_{0}\|_{L^{2}}$ for all $t\geq 0$.
\item[(ii)] Suppose moreover that $\Delta_{g}$ has only simple eigenvalues. Then, for almost every $p\in\mathcal{M}$, the pair $(\Delta_{g},\delta_{p})$ is approximately controllable, and consequently for every $\nu>0$ there exists $K\in\mathcal{L}(L^{2}(\mathcal{M}),\R)$ such that the same Cauchy problem is well-posed and satisfies $\|u(t)\|_{L^{2}}\leq Ce^{-\nu t}\|u_{0}\|_{L^{2}}$ for some $C>0$.
\end{itemize}
\end{prop}
\begin{proof}
By the Sobolev embedding, $\delta_{p}\in H^{-d/2-\varepsilon}(\mathcal{M})\subset H^{-2}(\mathcal{M})$, and since $(e_{n})_{n\geq 1}$ is a Hilbert basis of $L^{2}(\mathcal{M})$,
\begin{equation}
\delta_{p}=\sum_{n\geq 1} e_{n}(p)\,e_{n}\quad\text{in }H^{-d/2-\varepsilon}(\mathcal{M}).
\end{equation}
\emph{(i)} Since $\mathcal{M}$ is connected, $\lambda_{1}=0$ is simple with $e_{1}=1/\sqrt{\mathrm{Vol}(\mathcal{M})}$, hence $m(\lambda)=1$ for $\lambda\in(0,-\lambda_{2})$. Then $L_{\lambda}=\mathrm{span}(e_{1})$, $(\Delta_{g})_{L}=0$ and $(\delta_{p})_{L}=1/\sqrt{\mathrm{Vol}(\mathcal{M})}$. Solving the one-dimensional $F$-equivalence is trivial and yields $K=-\lambda\sqrt{\mathrm{Vol}(\mathcal{M})}\langle\cdot,e_{1}\rangle_{L^{2}}$, that is, the displayed formula. \emph{(ii)} The nodal set $e_{n}^{-1}(\{0\})$ has measure zero, so for almost every $p\in\mathcal{M}$, $e_{n}(p)\neq 0$ for all $n\geq 1$. Lemma~\ref{lem:Fattorini} then yields approximate controllability of $(\Delta_{g},\delta_{p})$, hence $\delta_{p}$ is $F_{\lambda}$-admissible (since $m(\lambda)=1$ under the simplicity assumption), and Theorem~\ref{thm:main} together with Proposition~\ref{prop:sg} concludes.
\end{proof}

\begin{remark}
As always with \textcolor{black}{parabolic} \(F\)-equivalence, the feedback \(K\) can be easily constructed using the same method as in the example following Corollary \ref{cor:stab-KS}.
\end{remark}

In the following subsections, we provide some concrete applications of Theorem \ref{thm:nonlinear-stab} to some classical PDE controlled systems.

\subsection{Kuramoto–Sivashinsky equation}\label{subsec:KS-eq} Here, we focus on the Kuramoto–Sivashinsky equation on the one-dimensional torus, which is given by
\begin{equation}\label{eq:KS-eq}
    \partial_t u + \Delta^2 u + \Delta u + \frac{1}{2} \partial_x (u^2) = 0.
\end{equation}
This equation was introduced by Yoshiki Kuramoto and Gregory Sivashinsky to study flame front propagation, for more details see \cite{kuramoto1978diffusion,sivashinsky1980on,sivashinsky1977nonlinear}. To apply Theorem \ref{thm:nonlinear-stab}, we need to establish the appropriate setting. We work in \( H = L^2(\mathbb{T}) \), and define \( A = -(\Delta^2 + \Delta) \), hence we have 
 \begin{equation}
     \forall s \in \R, \quad D_s(A) = H^{4s}(\T).
 \end{equation}
 The eigenbasis of \( A \) is \( (e_n)_{n \in \mathbb{Z}} \), defined as follows
 \begin{equation} 
   \forall n\in \Z, \forall \theta \in \T, \quad  e_n(\theta) = (2\pi)^{-1/2}e^{i n \theta }.
\end{equation}
 In order to apply our result we can reindex by \( \mathbb{N}_{>0} \) in the following way
\begin{equation}\label{eq:eigen-basis}
    \forall n \geq 1 , \ \tilde{e}_n := \begin{cases}
        e_k \ \text{if} \ n=2k+1, \, k \geq 0,\\
        e_{-k} \ \text{if} \ n=2k, \, k \geq 1.
    \end{cases} 
\end{equation}

With this notation we similarly define $\tilde{\lambda}_{n}$, we deduce from the previous section that
\begin{equation}
    \tilde{\lambda}_n = -\left\lfloor \frac{n}{2} \right\rfloor^4 + \left\lfloor \frac{n}{2} \right\rfloor^2 \, \underset{n \to +\infty}{\sim} -\frac{n^4}{16}.
\end{equation}
Hence $A$ is a diagonal parabolic operator on $H$.
For simplicity, we will continue to use the family \((e_n)_{n \in \mathbb{Z}}\) for the Sobolev norms.

We want to define \( \mathcal{F}\textcolor{black}{(u)}=-\frac{1}{2}\d_x(u^2) \) as a map from \( L^2(\mathbb{T}) \) to \( H^{-2}(\mathbb{T}) \), we will use the following lemma.
\begin{lem}
    Let \( u, v \in L^2(\mathbb{T}) \). Then there exists a constant \( C > 0 \) such that
    \begin{equation}
        ||\d_x(uv)||_{H^{-2}} \leq C ||u||_{L^2} ||v||_{L^2}.
    \end{equation}
\end{lem}
\begin{proof}
    Let \( u, v \in C^{\infty}(\T) \), we have
    \begin{equation}
        uv = \sum_{n \in \Z} \left( \sum_{k\in \Z} u_k v_{n-k}\right) e_n.
    \end{equation}
    We define $\<n\> = \sqrt{1+|n|^2}$ for all $n\in \Z$. Then by definition of Sobolev norms we have
    \begin{equation}
        ||\d_x(uv)||_{H^{-2}}^2 = \sum_{n \in \Z} n^ 2 \left| \sum_{k\in \Z} u_k v_{n-k}\right|^2 \<n\>^{-4}.
    \end{equation}
    Now, applying Cauchy-Schwarz inequality gives
    \begin{equation}
        ||\d_x(uv)||_{H^{-2}}^2  \leq ||u||_{L^2}^2 ||v||_{L^2}^2 \sum_{n \in \Z} \<n\>^{-2}.
    \end{equation}
    Which concludes the proof.
\end{proof}
Now, as $\textcolor{black}{\mathcal{F}}$ is quadratic, the previous lemma ensures us that $\textcolor{black}{\mathcal{F}}$ satisfies \textcolor{black}{Assumption} \ref{item:F1}. Note that for all $\lambda>0$, we have $m(\lambda)=3$. Hence, applying Theorem \ref{thm:nonlinear-stab} leads to the following immediate corollary.
\begin{cor}\label{cor:stab-KS}
Let \(\lambda \in \R_{>0}\). Suppose that \((f_1, f_2, f_3) \in (H^{-2}(\mathbb{T}))^3\) is an \(F_\lambda\)-admissible control operator. Then there exist \(K_1, K_2, K_3 \in \mathcal{L}(L^2(\mathbb{T}), \mathbb{C})\) and $\delta>0$ such that for every \(u_0 \in L^2(\mathbb{T})\) with $\|u_0\|_{L^2} \leq \delta$, there exists a unique maximal solution \(u(\cdot) \in C^0([0, +\infty); L^2(\mathbb{T})) \)
to 
\begin{equation}\label{eq:KS-stab}
    \begin{cases}
        \partial_t u + \Delta^2 u + \Delta u + \frac{1}{2}\partial_x (u^2) + f_1 K_1 u + f_2 K_2 u + f_3 K_3 u = 0, \\
        u(0) = u_0.
    \end{cases}
\end{equation}
Furthermore, there exist \(C_\lambda > 0\) such that
\begin{equation}
    \forall t \geq 0, \quad ||u(t)||_{L^2} \leq C_\lambda e^{-\lambda t} ||u_0||_{L^2}.
\end{equation}
\end{cor}
We now provide a concrete application of the above corollary to demonstrate that the feedback derived from the parabolic \(F\)-equivalence is easily constructible.

Suppose we want to stabilize the system at a rate $\lambda = 20$, then $N(\lambda)=5$. For this example, we define 
\begin{equation}
    \forall x \in \mathbb{T}, \quad f_1(x) = \frac{1}{\sqrt{2\pi}}, \quad f_2(x) = \frac{1}{\sqrt{2\pi}}(e^{-ix} + e^{-2ix}), \quad f_3(x) = \frac{1}{\sqrt{2\pi}}(e^{ix} + e^{2ix}).
\end{equation}
Thus, we have \(f_1 = \tilde{e}_1, \ f_2 = \tilde{e}_2 + \tilde{e}_4, \ f_3 = \tilde{e}_3 + \tilde{e}_5 \in L_\lambda\), and \(B = (f_1, f_2, f_3)\) is clearly \(F_\lambda\)-admissible.

To find our feedbacks, we only need to solve a finite-dimensional \(F\)-equivalence problem. Identifying \((\tilde{e}_1, \tilde{e}_2, \tilde{e}_4, \tilde{e}_3, \tilde{e}_5)\) with the canonical basis of \(\mathbb{C}^5\), we have
\begin{equation}
    A_L = \begin{pmatrix}
        0 & 0 & 0 & 0 & 0 \\
        0 & 0 & 0 & 0 & 0 \\
        0 & 0 & -12 & 0 & 0 \\
        0 & 0 & 0 & 0 & 0 \\
        0 & 0 & 0 & 0 & -12 \\
    \end{pmatrix}, \quad
    f_1 = \begin{pmatrix}
        1 \\
        0 \\
        0 \\
        0 \\
        0 
    \end{pmatrix}, \quad
    f_2 = \begin{pmatrix}
        0 \\
        1 \\
        1 \\
        0 \\
        0 
    \end{pmatrix}, \quad
    f_3 = \begin{pmatrix}
        0 \\
        0 \\
        0 \\
        1 \\
        1 
    \end{pmatrix}.
\end{equation}

Let \(\mu \geq \lambda\). If we denote by \((\tilde{T},\tilde{K})\) the solution of 
\begin{equation}
    \begin{cases}
        \tilde{T}(A_L + B\tilde{K}) = (A_L - \mu)\tilde{T}, \\
        T B = B,
    \end{cases}
\end{equation}
then since \(TB = B\) is equivalent to \(T f_1 = f_1\), \(T f_2 = f_2\), and \(T f_3 = f_3\), we can decompose the problem into three subproblems. It is straightforward to solve these either manually or numerically, and we find that
\begin{align}
    \begin{split}
        \tilde{T} &= \begin{pmatrix}
            1 & 0 & 0 & 0 & 0 \\
            0 & \frac{\mu}{12} + 1 & -\frac{\mu}{12} & 0 & 0 \\
            0 & \frac{\mu}{12} & 1 - \frac{\mu}{12} & 0 & 0 \\
            0 & 0 & 0 & \frac{\mu}{12} + 1 & -\frac{\mu}{12}  \\
            0 & 0 & 0 & \frac{\mu}{12} & 1 - \frac{\mu}{12}  \\
        \end{pmatrix}, \\
        \tilde{K} &= \begin{pmatrix}
            -\mu & 0 & 0 & 0 & 0 \\
            0 & -\frac{\mu(\mu + 12)}{12} & \frac{\mu(\mu - 12)}{12} & 0 & 0 \\ 
            0 & 0 & 0 & -\frac{\mu(\mu + 12)}{12} & \frac{\mu(\mu - 12)}{12} \\ 
        \end{pmatrix}.
    \end{split}
\end{align}
By Theorem \ref{thm:main}, we know that for almost every \(\mu \geq \lambda\), we can use \(\tilde{K}\) to define our feedbacks in Corollary \ref{cor:stab-KS}. Then for all \(f \in L^2(\mathbb{T})\), the above corollary applies with
\begin{align}
    \begin{split}
        K_1f &=  \frac{1}{\sqrt{2\pi}}\int_{\mathbb{T}} -\mu f(x) \, dx, \\
        K_2f &= \frac{1}{\sqrt{2\pi}}\int_{\mathbb{T}} f(x) \left(- \frac{\mu(\mu + 12)}{12} e^{ix} +  \frac{\mu(\mu - 12)}{12} e^{2ix}\right) \, dx, \\
        K_3f &= \frac{1}{\sqrt{2\pi}}\int_{\mathbb{T}} f(x) \left(- \frac{\mu(\mu + 12)}{12} e^{-ix} +  \frac{\mu(\mu - 12)}{12} e^{-2ix}\right) \, dx.
    \end{split}
\end{align}

\textcolor{black}{As we will now briefly show, our result can also be applied to} the Kuramoto-Sivashinsky system studied in \cite{CoronLu15}, as it was one of the first examples of Fredholm backstepping, that is
\begin{equation}\label{eq:system-coronlu}
\left\{
\begin{aligned}
    \partial_t u + \Delta^2 u + \nu \Delta u + \tfrac{1}{2}\,\partial_x(u^2) &= 0
        && \text{in } (0,1) \times (0,+\infty),\\
    u(t,0)=u(t,1) &= 0
        && \text{for all } t>0,\\
    \Delta u(t,0)=w(t),\quad \Delta u(t,1) &= 0
        && \text{for all } t>0,\\
    u(0,\cdot) &= u_0
        && \text{in } L^2(0,1).
\end{aligned}
\right.
\end{equation}

with $\nu > 0$. Thus, we work in $H=L^2(0,1)$, with $A=-\Delta^2 - \nu \Delta$ and 
\begin{equation*}
    D(A)=\{ u \in H^4(0,1) \mid u(0)=u(1)=\Delta u(0)=\Delta u(1)=0 \}.
\end{equation*}
Setting $e_n(x) = \sqrt{2} \sin(\pi n x)$ for all $x \in (0,1)$ and for $n \geq 1$, we observe that $(e_n)_{n \geq 1}$ forms an orthonormal basis of $H$ consisting of eigenvectors of $A$, with eigenvalues
\begin{equation}
   \forall n \geq 1, \quad \lambda_n = -\pi^4 n^4 + \nu \pi^2 n^2.
\end{equation}
Thus, $A$ is a self-adjoint diagonal parabolic operator on $H$.
Multiplying \eqref{eq:system-coronlu} by a smooth function in $D(A)$ and integrating by parts, we obtain 
\begin{equation}
    \forall u \in D(A), \quad B^* u = -\d_x u(0).
\end{equation}
which defines, by duality, $B \in \mathcal{L}(\R, D(A)')$. Moreover, we have 
\begin{equation}\label{eq:coronlu-B}
    \forall n \geq 1, \quad \scalp{B}{e_n}_{D(A)',D(A)} = -\pi n.
\end{equation}
Following \cite{CoronLu15}, we assume that 
\begin{equation}
    \nu \not \in \{n^2\pi^2 + k^2 \pi^2 \mid n,k \geq 1, \, n \neq k \},
\end{equation}
which ensures that $A$ has only simple eigenvalues (thus, $(A,B)$ is approximately controllable by Lemma \ref{lem:Fattorini}). 
Consequently, \eqref{eq:coronlu-B} implies that $B$ is $F_\lambda$-admissible for all $\lambda > 0$. The nonlinearity $\mathcal{F}$ can be handled similarly to the previous example. Therefore, applying Theorem \ref{thm:nonlinear-stab}, we obtain the following corollary.
\begin{cor}
Let \(\lambda \in \R_{>0}\). There exists \(K \in \mathcal{L}(L^2(0,1), \mathbb{R})\) and $\delta > 0$ such that for every \(u_0 \in L^2(0,1)\) with $\|u_0\|_{L^2} \leq \delta$,  there exists a unique maximal solution \(u \in C^0([0, +\infty); L^2(0,1)) \) to 
\begin{equation}\label{eq:KS-stab-2}
\left\{
\begin{aligned}
    \partial_t u + \Delta^2 u + \nu \Delta u + \tfrac{1}{2}\,\partial_x(u^2) &= 0
        && \text{in } (0,1) \times (0,+\infty),\\
    u(t,0)=u(t,1) &= 0
        && \forall  t > 0,\\
    \Delta u(t,0)=K u(t,\cdot),\quad \Delta u(t,1) &= 0
        && \forall t > 0,\\
    u(0,\cdot) &= u_0
        && \text{in } L^2(0,1).
\end{aligned}
\right.
\end{equation}

Furthermore, there exist constants \(C_\lambda > 0\) such that
\begin{equation}
   \forall t \geq 0, \quad ||u(t, \cdot)||_{L^2} \leq C_\lambda e^{-{\lambda} t} ||u_0||_{L^2}.
\end{equation}
\end{cor}
\begin{remark}
    As before, the feedback $K$ is constructed by solving a finite-dimensional linear system. Compared to \cite{CoronLu15}, this approach is significantly simpler. \textcolor{black}{One could similarly use our approach for the Kuramoto-Sivashinsky system found in \cite{liu2001stability}.}
\end{remark}

\subsection{Navier-Stokes equations}\label{subsec:Navier-Stokes}
Let $\Omega$ be a smooth bounded domain in $\mathbb{R}^2$. In this subsection we consider the scalar controlled Navier-Stokes equations which are given by
\begin{equation}\label{eq:Navier-Stokes}
\left\{
\begin{aligned}
    \partial_t u - \Delta u + (u \cdot \nabla) u + \nabla p &= \sum_{i=1}^m w_i(t) f_i
        && \text{in } \Omega \times (0,+\infty)  ,\\
    \nabla \cdot u &= 0 
        && \text{in }  \Omega \times  (0,+\infty),\\
    u &= 0 
        && \text{on }  \partial\Omega \times (0,+\infty),\\
    u(0, \cdot) &= u_0(\cdot) 
        && \text{in } \Omega.
\end{aligned}
\right.
\end{equation}
We first need to express \eqref{eq:Navier-Stokes} in our framework. In order to do this we will reuse classical results from the literature about Navier-Stokes systems, such as the ones found in \cite{Raymond2007}.
The literature on the stabilization of Navier-Stokes equations is extensive, see, for example \cite{badra2020local, badra2011stabilization, badra2014fattorini,BarbuTriggiani2004IUMJ,Raymond2006SICON,Raymond2007JMPA,BarbuLasieckaTriggiani2006Memoirs,Badra2009COCV_Extended,breiten2019feedback,RaymondThevenet2010DCDS,AzouaniTiti2014EECT, Mitra2019COCV_Nonhomogeneous, BuffeTakahashi2023CRMath_FSI,xiang2023small} among many others.

We introduce the following functional spaces
\begin{align}
    L^2_{\sigma}(\Omega) &= \overline{\{ u \in C^{\infty}_c(\Omega;\R^2) \ | \ \nabla \cdot u = 0 \text{ in } \Omega \}}^{L^2},\\
    V_0^s(\Omega) &= \{ u \in H^s(\Omega;\R^2) \ | \ \nabla \cdot u = 0 \text{ in } \Omega, \ u = 0 \text{ on } \partial\Omega \}, \ s > \frac{1}{2}.
\end{align}
Recall the Leray-Helmholtz decomposition $L^2(\Omega;\R^2) = L^2_{\sigma}(\Omega) \oplus \nabla H^1(\Omega)$. We denote by $\P$ the orthogonal projection onto $L^2_{\sigma}(\Omega)$, which is often called the Leray projection. 
Our state space will be $H=L^2_{\sigma}(\Omega)$, and our operator will be the Stokes operator defined by
\begin{equation}
    A= \P \Delta \ \text{with} \ D(A) = V_0^2(\Omega).
\end{equation}
It is well-known that \(A\) is a self-adjoint diagonal parabolic operator on \(H\) with negative eigenvalues, see \cite[Lemma 3.1]{Fursikov2001}.

Here and in the following, every element of $H^{-1}(\Omega;\R^2)$ is seen as an element in $D_{-1/2}(A)$ by the continuous extension of the Leray projector from $H^{-1}(\Omega;\R^2)$ to $D_{-1/2}(A)$. Now we consider the nonlinearity given by 
\begin{equation}
\label{eq:FNS}
\F(u) = (u \cdot \nabla) u.
\end{equation}

From \cite[Section 5]{badra2011stabilization} \textcolor{black}{the system \eqref{eq:Navier-Stokes} is equivalent to 
\begin{equation}
    \begin{cases}
        \partial_t u - Au + \mathcal{F}(u) = 0, \\
        u(0) = u_0 \in H.
    \end{cases}
\end{equation} 
Also from \cite[Section 5]{badra2011stabilization}},
$\mathcal{F}$ goes from $D_{1/2}(A)=V^1_0(\Omega)$ to $H^{-1}(\Omega;\R^2)$, and 
\begin{equation}
\forall u,v \in V^1_0(\Omega), \quad \|(u \cdot \nabla) v\|_{H^{-1}}
    \le C\|u\|_{L^2} \|v\|_{H^1}. 
\end{equation}

Here $B=(f_1,\dots,f_m)$, hence from Remark \ref{rem:NL-sufficient}, we deduce that for any control operator such that $B \in D_{-1/2}(A)^m$ (for instance if $B \in H^{-1}(\Omega;\R^2)^{\textcolor{black}{m}}$), $\mathcal{F}$ satisfies Assumption \ref{item:F1}.


Now for the sake of doing a simple illustration of Theorem \ref{thm:nonlinear-stab}, we will assume that $\Omega$ is such that the Stokes operator has a simple spectrum (from \cite{OrtegaZuazua2001}, we know it is generically the case). Hence
we only need one forcing term as a control operator, which we denote by \(f \in H^{-1}(\Omega;\R^2)\). We will assume that \((A,f)\) is approximately controllable, and thanks to Lemma \ref{lem:Fattorini}, if we denote by $(e_n)_{n\geq 1}$ the eigenfunctions of \(A\), this is equivalent to the following condition
\begin{equation}\label{eq:approxNS}
    \forall n \geq 1, \quad \langle f, e_n \rangle_{H^{-1},H^1} \neq 0.
\end{equation}
The next proposition summarizes the stabilization result in the described setting.
\begin{prop}
\textcolor{black}{Let $\Omega$ be a smooth bounded domain of $\mathbb{R}^{2}$ such that} the Stokes operator has a simple spectrum. Let $f\in H^{-1}(\Omega;\R^2)$ 
\textcolor{black}{such that}
$(A,f)$ is approximately controllable. (i.e. \eqref{eq:approxNS} holds)
Then for all $\lambda>0$, there exists a feedback operator
 $K \in \mathcal{L}(L^2_{\sigma}(\Omega),\R)$ and $\delta>0$, such that for every $u_0 \in L^2_{\sigma}(\Omega)$ with $\|u_0\|_{L^2} \leq \delta$, there exists a unique solution \(u \in C^0_b([0, +\infty); L^2_{\sigma}(\Omega)) \cap L^2((0, +\infty); V^1_0(\Omega))\) to
\begin{equation}\label{eq:Navier-Stokes-stab}
    \begin{cases}
        \partial_t u - \P \Delta u + (u \cdot \nabla) u = (K u)f, \\
        u(0) = u_0.
    \end{cases}
\end{equation} 
Moreover, there exists a constant \(C > 0\) (independent of $u_0$) such that
\begin{equation}
    \forall t\geq 0, \quad \|u(t)\|_{L^2} \leq C e^{-{\lambda}t}\|u_0\|_{L^2}.
\end{equation}
\end{prop}
\begin{remark}
    Note that again, the strength of Assumption \ref{item:F1} described in Remark \ref{rem:NL-novelty} allows us to have irregular forcing terms, without any additional works.
\end{remark}
\subsection{Quasilinear heat equation}\label{subsec:quasilinear-heat}
Let $\Omega$ be a smooth connected bounded open subset of $\R^d$, with $d=2$.\footnote{We choose $d=2$ to have sharp assumptions on $f$ and $D$.} In this subsection we consider a general form of quasilinear heat equation given by
\begin{equation}\label{eq:quasilinear-heat-control}
\left\{
\begin{aligned}
    \partial_t u  - \div(D(u)\nabla u) &=  f(u) + w(t) b
        && \text{in } \Omega \times (0,+\infty)  ,\\
        D(0) \nabla u \cdot n &= 0 && \text{on } \partial \Omega \times (0,+\infty),\\
    u(0, \cdot) &= u_0(\cdot) 
        && \text{in } \Omega.
\end{aligned}
\right.
\end{equation}
We consider here a classical quasilinear diffusion model. While the literature on feedback stabilization for semilinear parabolic PDEs is extensive, 
results for quasilinear dynamics are comparatively scarce, see, for instance, the recent 1D stabilization result for a quasilinear heat equation \cite{BelhadjoudjaMaghenemWitrantKrstic_2025_arXiv250510935}. In order to keep the presentation simple and the feedback fully explicit (proportional to the spatial average), we restrict ourselves to a single scalar control. This already provides exponential stabilization under the assumptions below. Of course, if one needs rapid stabilization, the approach readily extends by adding more \textcolor{black}{controls}, which allows one to accelerate the decay rate at will.

We set $\widetilde{D} = D - D(0)$ and $\widetilde{f}(u)=f(u)-f'(0)u$. Let $\varepsilon \in (0,1/2)$ and set $s= d/2 +\varepsilon$, \textcolor{black}{we make the following assumptions} on $D$ and $f$:
\begin{enumerate}[label=(D\arabic*)]
    \item \label{item:D1} We have $D \in C^{2,1}_{\text{loc}}(\R;\R^{d \times d})$ with $D(0)$ symmetric positive definite.
    \item \label{item:D2} We have $f \in C^{1,1}_{\text{loc}}(\R)$ with $f(0)=0$.
    \item  \label{item:D3} We have $b \in H^{\varepsilon}(\Omega)$ and $\scalp{b}{1}_{L^2} \not = 0$.
\end{enumerate}
\textcolor{black}{Here} $C^{1,1}_{\text{loc}}$ means continuously differentiable with locally Lipschitz derivative \textcolor{black}{and $C^{2,1}_{\text{loc}}$ is defined similarly.}
Now we express \eqref{eq:quasilinear-heat-control} in our framework. First\textcolor{black}{,} to define $A$, we define the intermediate operator $\textcolor{black}{A_{0}}=\div(D(0)\nabla(\cdot))$ on $L^2(\Omega)$ with the usual Neumann boundary condition domain. 
It is well known (see \cite[Chapter 9]{Brezis2010}) that $A_0$ is self-adjoint and diagonal parabolic on $L^2(\Omega)$. We denote by $(\lambda_n')_{n\geq 0}$ its eigenvalues and $(e_n')_{n\geq 0}$ an associated orthonormal basis of eigenvectors, notice that $\lambda_0' = 0$ and $e_0' = 1/{\sqrt{|\Omega|}}$. We will work with the state space $H= D_s(\textcolor{black}{A_{0}}) = H^{s}(\Omega)$ \textcolor{black}{(note that,} as $\varepsilon<1/2$, $s<\frac{3}{2}$ and there is no Neumann boundary condition to add). 
\textcolor{black}{W}e define $A$, as the operator with domain $D(A) = D_{s+1}(\textcolor{black}{A_{0}})$ \textcolor{black}{th}at act as follows
\begin{equation}
   \forall u \in D(A), \quad  Au=\textcolor{black}{A_{0}}u+f'(0)u.
\end{equation}
Hence $A$ is self-adjoint and diagonal parabolic on $H$, and its eigenvalues are $\lambda_n = \lambda_n' + f'(0)$, associated to the eigenvectors $e_n=e_n'$, for $n\geq 0$. We assume that $f'(0) \geq 0$, \textcolor{black}{
    such that the uncontrolled system is not exponentially stable.}

Now, here $B=b$ and by \ref{item:D3} we \textcolor{black}{have} $B \in H^{\varepsilon}(\Omega)=D_{-1/2}(A)$\textcolor{black}{. 
Here}
$\mathcal{F}$ is defined 
on $D_{1/2}(A)=\{ u \in H^{s+1}(\Omega) \ | \ D(0)\nabla u \cdot n = 0 \text{ in } \partial \Omega \}$,
by
\begin{equation}\label{eq:ql-nl-defF}
    \mathcal{F}(u) = \div(\widetilde{D}(u)\nabla u) + \widetilde{f}(u).
\end{equation}
\textcolor{black}{
We have the following Lemma, shown in Section \ref{sec:QHE}:
    \begin{lem}
    \label{lem:FQHE}
    $\mathcal{F}$ satisfies Assumption \ref{item:F1}.
    \end{lem}
} Now we define the following feedback for $\mu>0$
    \begin{equation}\label{eq:ql-nl-fb}
        \forall u \in H, \quad Ku = -\frac{\mu}{\scalp{b}{1}_{L^2}} {\int_\Omega u}.
    \end{equation}
    We are now able to sta\textcolor{black}{te} 
    the main result of this subsection.
    \begin{prop}\label{prop:ql-main}
        Assume that $0\leq f'(0) < |\lambda_1|$ and let $\lambda \in (0, |\lambda_1| - f'(0))$. Then, there exists 
        $\delta >0$ such that for every $u_0 \in H^s(\Omega)$ with $\|u_0\|_{H^s} \leq \delta$, there exists a unique solution $u \in C^0_b([0,+\infty);H^s(\Omega)) \cap L^2((0,+\infty);H^{s+1}(\Omega))$ to
\begin{equation}\label{eq:ql-eq-main}
\left\{
\begin{aligned}
    \partial_t u  - \div(D(u)\nabla u) &=  f(u) + K(u) b
        && \text{in } \Omega \times (0,+\infty)  ,\\
        D(0) \nabla u \cdot n &= 0 && \text{on } \partial \Omega \times (0,+\infty),\\
    u(0, \cdot) &= u_0(\cdot) 
        && \text{in } \Omega,
\end{aligned}
\right.
\end{equation}
where $K$ is defined in \eqref{eq:ql-nl-fb} with a $\mu \geq \lambda + f'(0)$. Moreover, it is exponentially stable, meaning that there exists a constant $C > 0$ such that
\begin{equation}
    \forall t\geq 0, \quad \|u(t)\|_{H^s} \leq C e^{-\lambda t} \|u_0\|_{H^s}.
\end{equation}
    \end{prop}
    \begin{proof}
    
        Here we have $L_\lambda = \text{span} \{e_0\}$ and $m(\lambda)=1$ as $\lambda_0$ is simple. Hence \ref{item:D3} ensures that $B$ is $F_\lambda$-admissible \textcolor{black}{(note that $\langle B,1 \rangle_{H^{s}}$ is directly deduced from $\langle B,1 \rangle_{L^{2}}$)}. Hence, \textcolor{black}{from Lemma \ref{lem:FQHE},} there exists an $F$-equivalence by Theorem \ref{thm:main}, and projecting $T(A+BK)=DT$ on $e_0$, gives
        \begin{equation}
            K(u) \scalp{B}{e_0}_{H^s} = -\mu\scalp{u}{e_0}_{H^s}.
        \end{equation}
        Hence the $F$-equivalence feedback is given by $K$ defined in \eqref{eq:ql-nl-fb} with some $\mu \geq \lambda + f'(0)$.
        Then it suffices to apply Theorem \ref{thm:nonlinear-stab} to ensure local well-posedness and exponential stability.
    \end{proof}

\section{Main proofs}\label{sec:MainProof} 
\subsection{Existence of parabolic $F$-equivalence}\label{subsec:feq-mainproof} In this subsection we prove Theorem \ref{thm:main}.
\subsubsection{Proof strategy} 
Let \(A\), \(\lambda\), \(D\), and \(B\) be as in Theorem \ref{thm:main}. In particular, our frequency decomposition of spaces are always with respect to \(\lambda\). Below, we briefly outline the proof steps for Theorem \ref{thm:main}:

\begin{enumerate}
    \item We begin by establishing necessary conditions on the form of \((T, K)\) for it to be a parabolic \(F\)-equivalence of \((A, B, D)\), as detailed in Proposition \ref{prop:partial-uniqueness}.
    
    \item In Subsection \ref{subsec:simple}, we first apply Lemma \ref{lem:decomp1} using the partition induced by the \(F_{\lambda}\)-admissibility of \(B\). Hence, for each \(j \in \{1, \dots, m(\lambda)\}\), we have \(A_j\), a diagonal parabolic operator on \(\mathcal{H}^j\), with \(B_j \in D(A_j)'\). Exploiting the fact that \(A_j\) has only simple eigenvalues \textcolor{black}{(see, for instance, \cite[(2.10)]{CoronLu15})} in $L_\lambda^j$, we show in Proposition \ref{prop:simple-feq} that for almost every \(\mu \geq \lambda + c_A\) and for all \(j \in \{1, \dots, m(\lambda)\}\), there exists a parabolic \(F\)-equivalence \((T_j, K_j)\) between \((A_j, B_j)\) and \(D_j\), where (\textcolor{black}{${P_j}_L$ and ${P_j}_H$ are the orthogonal projections on $L_\lambda^j$ and $\mathcal{H}_\lambda^j$ in $\mathcal{H}^j$})
\textcolor{black}{\begin{equation}
    D_j = ({A_j}_L - \mu) {P_j}_L + {A_j}_H {P_j}_H.
\end{equation}}
    The core of the proof lies in this step, with the main technical challenge being to establish that \( T_j \) is an isomorphism. To achieve this, we make essential use of the polynomial properties of finite-dimensional \( F \)-equivalence feedback, see Theorem \ref{thm:ffeq}. Note that \(D_j\) explicitly depends on \(\mu\), but for notational convenience, we do not indicate this dependence explicitly.
    
    \item Finally, in Subsection \ref{subsec:conc}, we prove Theorem \ref{thm:main}. To this end we set
\begin{equation}
    T = T_1 + \dots + T_{m(\lambda)}, \quad K = (K_1, \dots, K_{m(\lambda)}).
\end{equation}
Then we demonstrate that \((T, K)\) indeed forms a parabolic \(F\)-equivalence between \((A, B)\) and \(D\).
\end{enumerate}

\subsubsection{\textcolor{black}{Necessary conditions on $(T,K)$}}

\textcolor{black}{The goal of this subsection is to} show some conditions that \((T, K)\) should satisfy to be a parabolic \(F\)-equivalence between \((A, B)\) and \(D\) \textcolor{black}{(see Proposition \ref{prop:partial-uniqueness})}. This will greatly help us understand the form of \((T, K)\) in Subsection \ref{subsec:simple}, and we will also reuse it in Section \ref{section:wfeq}.\\

\textcolor{black}{First, let us introduce the following notation: w}e denote by \(P_{L}\) and \(P_{H}\) the orthogonal projections on \(L_\lambda\) and \(H_\lambda\) in \(H\), and by \textcolor{black}{a slight} abuse of notation, we use the same symbols for the orthogonal projections on \(L_\lambda\) and \(D(A)'_\lambda\) in \(D(A)'\).

Then for every \(x \in D(A)'\), we have \(x = x_L + x_H\) with \(x_L = P_L x\) and \(x_H = P_H x\). Now, for every normed vector space \(E\), the previous decompositions give us the following decomposition on the space of bounded operators
\[ 
\mathcal{L}(H, E) =\mathcal{L}(L_\lambda, E) \oplus \mathcal{L}(H_\lambda, E). 
\]
This is also true when replacing \(H\) with \(D(A)'\). Notice that $A(L_{\lambda}) \subset L_{\lambda}$ and  that $A(H_\lambda \cap D(A)) \subset H_\lambda$, so similarly we can define $A_L = A P_L$ and $A_H = A P_H$, and we have $A = A_L + A_H$.\\

We now decompose $\mathcal{L}(H, D(A)')$ in terms of frequency. Let $M \in \mathcal{L}(H, D(A)')$. We can write
\begin{equation}
    M = P_{L} M P_{L} + P_{L} M P_{H} + P_{H} M P_{L} + P_{H} M P_{H},
\end{equation}
which allows us to define the direct sum corresponding to the above decomposition
\begin{equation}
    \mathcal{L}(H, D(A)') = LL_\lambda(H, D(A)') \oplus HL_\lambda(H, D(A)') \oplus LH_\lambda(H, D(A)') \oplus HH_\lambda(H, D(A)'), 
\end{equation}
\textcolor{black}{where, for instance,
\begin{equation}
HL_\lambda(H, D(A)') = \{ M \in \L(H, D(A)') \ | \ M = P_{L}MP_{H}\},
\end{equation}
and $LL_\lambda(H, D(A)')$, $LH_\lambda(H, D(A)') $ and $HH_\lambda(H, D(A)')$ are defined accordingly.}
We could apply the same approach to \(\mathcal{L}(H)\) or $\L({D(A)'}) $  and the subspaces defined earlier. In these cases, this decomposition allows us to use a matrix formalism. For example, if \(M \in \mathcal{L}(H)\), we write 
\begin{equation}
    M = \begin{pmatrix} M_{LL} & M_{HL} \\ 
                        M_{LH} & M_{HH} \end{pmatrix}, 
\end{equation}
with 
\[ 
M_{LL} = P_L M P_L, \quad M_{HL} = P_L M P_H, \quad M_{LH} = P_H M P_L, \quad M_{HH} = P_H M P_H,
\]
and we verify that the usual matrix multiplication rules apply. For instance, for \(x \in H\), we write $
    x = \begin{pmatrix}
    x_L \\ x_H
\end{pmatrix}$, and we have \[ Mx = M_{LL}x_L + M_{HL}x_H + M_{LH}x_L + M_{HH}x_H =\begin{pmatrix} M_{LL} & M_{HL} \\ 
M_{LH} & M_{HH}\end{pmatrix}  \begin{pmatrix}
    x_L \\ x_H
\end{pmatrix} .\]

Notice that we can define \(\mathcal{L}_{H_\lambda}(D(A)'_\lambda)\) as in Proposition \ref{prop:algebra}. The algebra defined in the proposition below arises naturally in parabolic \(F\)-equivalence, as we will see in Sections \ref{sec:MainProof} and \ref{section:wfeq}.

\begin{prop}
 We define the commutant algebra of \(A_H\) as
\begin{equation}\label{eq:commuting-def}
    C(A_H) = \{ M \in \mathcal{L}_{H_\lambda}(D(A)'_\lambda) \ | \ A_H M = M A_H \ \textnormal{in} \ \mathcal{L}(H_\lambda, D(A)'_\lambda) \}.
\end{equation}
It is a sub-Banach algebra of \(\mathcal{L}_{H_\lambda}(D(A)'_\lambda)\).
\end{prop}

\begin{proof}
It's clear that \(C(A_H)\) is a subalgebra of \(\mathcal{L}_{H_\lambda}(D(A)'_\lambda)\). We just need to show that it is closed.
Let \((M_n)_{n \geq 1}\) be a sequence in \(\mathcal{L}_{H_\lambda}(D(A)'_\lambda)\) such that \(M_n \rightarrow M\) in \(\mathcal{L}_{H_\lambda}(D(A)'_\lambda)\). 
Recall that the norm of $\mathcal{L}_{H_\lambda}(D(A)'_\lambda)$ is equivalent to the following
\begin{equation}
    \forall M \in \mathcal{L}_{H_\lambda}(D(A)'_\lambda), \ |||M||| := \|M\|_{\mathcal{L}(D(A)'_\lambda)} + \|M_{|H_\lambda}\|_{\mathcal{L}(H_\lambda)}.
\end{equation}
We have
\[
\|MA_H - A_HM\|_{\mathcal{L}(H_\lambda, D(A)'_\lambda)} \leq \| MA_H - M_nA_H\|_{\mathcal{L}(H_\lambda, D(A)'_\lambda)} + \| M_nA_H - A_HM\|_{\mathcal{L}(H_\lambda, D(A)'_\lambda)},
\]
for the first term notice that
\[
\| MA_H - M_nA_H\|_{\mathcal{L}(H_\lambda, D(A)'_\lambda)} \leq \|A_H\|_{\mathcal{L}(H_\lambda, D(A)'_\lambda)} \| M - M_n\|_{\mathcal{L}( D(A)'_\lambda)}.
\]
Now using that $M_nA_H=A_H M_n$, we get
\[
\| M_nA_H - A_HM\|_{\mathcal{L}(H_\lambda, D(A)'_\lambda)} \leq 
\|A_H\|_{\mathcal{L}(H_\lambda, D(A)'_\lambda)} \| M - M_n\|_{\mathcal{L}( H_\lambda)}.
\]
Then passing to the limit, we finally get $\|MA_H - A_HM\|_{\mathcal{L}(H_\lambda, D(A)'_\lambda)} = 0$, and hence $M \in C(A_H)$.

\end{proof}

\textcolor{black}{We can now show the following necessary conditions on $(T,K)$:}

\begin{prop}\label{prop:partial-uniqueness}
Let \(\mu \in \R_{> 0} \setminus \{\lambda_l - \lambda_h\}_{h \geq l \geq 1}\), we set
\[ 
D_\mu = \begin{pmatrix}
    A_L - \mu & 0 \\
    0 & A_H
\end{pmatrix}.
\]
Let \((T, K)\) be a parabolic \(F\)-equivalence between \((A, B)\) and \(D\) \textcolor{black}{(h}ence $K \in \L(L_\lambda,\C^{\ml})$\textcolor{black}{)}. Then, if we denote by \((\tilde{T}, \tilde{K})\) the \textcolor{black}{unique} finite-dimensional \(F\)-equivalence between \((A_L, B_L)\) and \(A_L - \mu\) given by Theorem \ref{thm:ffeq},\footnote{We can apply this theorem using the isomorphism between \(\C^{N(\lambda)}\) and \(L_\lambda\), which sends the canonical basis to \((e_n)_{1 \leq n \leq N(\lambda)}\).} we have
\begin{equation}\label{eq:TK-shape}
    T = \begin{pmatrix}
        \tilde{T} & 0 \\
        \tau & C
    \end{pmatrix}, \quad K = \tilde{K},
\end{equation}
with \(C \in C(A_H)\) and \(\tau \in LH_\lambda(D(A)')\) \textcolor{black}{and is defined as}
\begin{equation}
    \forall n \in \{1, \dots, N(\lambda)\}, \ \tau(e_n) = \sum_{k \geq 1} \frac{\scalp{B_H K e_n}{\sqrt{1 + |\lambda_k|^2} e_k}_{D(A)'}}{\lambda_k - \lambda_n} \sqrt{1 + |\lambda_k|^2} e_k.
\end{equation}
Furthermore, we have  \(\tau(L_\lambda) \subset H_\lambda\).

\end{prop}
\begin{proof}
Let \(n > N(\lambda)\). Then \(Ke_n = 0\), and by \(F\)-equivalence, we have
\begin{equation}\label{eq:lf}
    T(A + BK)e_n = \lambda_n Te_n = DTe_n.
\end{equation}
Hence, \(Te_n\) is an eigenvector of \(D\) \textcolor{black}{associated to the eigenvalue $\lambda_{n}$}. 
\textcolor{black}{Note that \((e_j)_{j \geq 1  }\) form a basis of eigenvectors of $D$, from its definition.}
Now, because \(\mu \notin \{\lambda_l - \lambda_h\}_{h \geq l \geq 1}\), we have $Te_n \in \text{span} (e_k)_{k>N(\lambda)} $. Otherwise, it should exist $l \leq N(\lambda)$ such that $\scalp{Te_n}{e_l}_H \not = 0$, but by \eqref{eq:lf} we would have
\[
\lambda_n \scalp{Te_n}{e_l}_H = (\lambda_l-\mu) \scalp{Te_n}{e_l}_H,
\]
which would be \textcolor{black}{a contradiction}. Thus, we have
\begin{equation}\label{eq:commute}
    \forall v \in H_\lambda, \ TA_Hv = A_HTv.
\end{equation}
Therefore, \(T_{HL} = 0\) and \(T_{HH} \in C(A_H)\).

Now let \(n \leq N(\lambda)\). Applying this to the \(F\)-equivalence equation and using $TB=B$ gives us
\begin{equation}
    \lambda_n Te_n + B_LKe_n + B_HKe_n = DTe_n.
\end{equation}
Projecting this on \(L_\lambda\) and \(D(A)'_\lambda\) gives us
\begin{equation}\label{eq:decomp-HL}
    \begin{cases}
        \lambda_n (Te_n)_L + B_LKe_n = (A_L - \mu)(Te_n)_L, \\
        \lambda_n (Te_n)_H + B_HKe_n = A_H (Te_n)_H.
    \end{cases}    
\end{equation}
Now, as $T_{HL}=0$, we have $T_{LL} B_L = B_L$ and hence the first equation above is equivalent to
\begin{equation}
    T_{LL}(A_L + B_LK) = (A_L - \mu)T_{LL}.
\end{equation}
Then, identifying \(L_\lambda\) with \(\C^{N(\lambda)}\) using the isomorphism that sends the canonical basis to \((e_n)_{1 \leq n \leq N(\lambda)}\), we have, \textcolor{black}{from Theorem \ref{thm:ffeq},} \(T_{LL} = \tilde{T}\) and \(K = \tilde{K}\). Now let's use the second equality in (\ref{eq:decomp-HL}). We set \(h_n := (Te_n)_H\). \textcolor{black}{Note that, given the choice of $n$, $h_{n} = T_{LH}e_{n}$.} For \(k > N(\lambda)\), we have
\begin{equation}
    \lambda_n \scalp{h_n}{e_k}_{D(A)'} + \scalp{B_HKe_n}{e_k}_{D(A)'} = \scalp{A_H h_n}{e_k}_{D(A)'}.
\end{equation}
Using \(A^* e_k = \overline{\lambda_k} e_k\), we get
\begin{equation}
    \scalp{h_n}{e_k}_{D(A)'} = \frac{\scalp{B_H Ke_n}{e_k}_{D(A)'}}{\lambda_k - \lambda_n}.
\end{equation}
Notice that by definition of $N(\lambda)$, we have \( \Re (\lambda_{N(\lambda) + 1} - \lambda_{N(\lambda)} ) \not = 0\), hence the above expression is well-defined \textcolor{black}{(recall that $k >N(\lambda)\geq n$)}. Recall that \((\sqrt{1 + |\lambda_m|^2} e_m)_{m \geq 1}\) is an orthonormal basis of \(D(A)'\), and so
\[ \tau(e_{n}) :=  h_n = \sum_{k \geq 1} \frac{\scalp{B_H K e_n}{\sqrt{1 + |\lambda_k|^2} e_k}_{D(A)'}}{\lambda_k - \lambda_n} \sqrt{1 + |\lambda_k|^2} e_k. \]
Now, to prove that \(\tau(L_\lambda) \subset H\), it suffices to show that \(\tau(e_n) \in H\) for each \(n \in \{1, \dots, N(\lambda)\}\).

Let \(k \geq 1\). We set \(\textcolor{black}{I}_k := \scalp{B_H K e_n}{\sqrt{1 + |\lambda_k|^2} e_k}_{D(A)'}\). This forms an \(\ell^2\) sequence \textcolor{black}{since $B_{H}\in D(A)'$}, and we have
\begin{equation}
    \frac{|\textcolor{black}{I}_k|}{|\lambda_k - \lambda_n|} \sqrt{1 + |\lambda_k|^2} \underset{k \rightarrow +\infty}{\sim} |\textcolor{black}{I}_k|.
\end{equation}
Thus \(\tau(e_n) \in H\).

\end{proof}

\begin{remark}
Here we emphasize that \(T_{HL} = 0\) reflects the internal structure of the equation
\begin{equation}
    \d_t u = Au + BKu,
\end{equation}
with \(K \in \mathcal{L}(L_\lambda, \C^{m(\lambda)})\). Notice that because $Kx=Kx_L$, this equation can be decoupled as follows:
\begin{equation}
    \d_t u = Au + BKu \iff \begin{cases}
        \d_t u_L = A_L u_L + B_L K u_L, \\
        \d_t u_H = A_H u_H + B_H K u_L.
    \end{cases}
\end{equation}
The first evolution equation on \(u_L\) in the system is independent of \(u_H\) and can be solved on its own. This fact is then reflected by \(T_{HL} = 0\) in the \(F\)-equivalence. Similarly, one can observe that \(T_{LH} \neq 0\) if \(B_H \neq 0\), as the low-frequency part \(u_L\) influences the evolution of \(u_H\).

\end{remark}

\begin{remark}
Notice that \(K\) is entirely determined \textcolor{black}{as soon as $A$, $B$ and $D$ are given. $T$ is not, however, and the only free} component of \(T\) is \(C\). As we will discuss in Section \ref{section:wfeq}, this is \textcolor{black}{what} prevents achieving uniqueness.

\end{remark}

\subsubsection{Simple multiplicity case}\label{subsec:simple}
In this section, we aim to establish the following proposition.

\begin{prop}\label{prop:simple-feq}
    For almost every \(\mu \geq \lambda + c_A\) and for all \(j \in \{1,\dots, m(\lambda)\}\), there exists a parabolic \(F\)-equivalence \((T_j, K_j)\) between \((A_j, B_j)\) and the \(\lambda\)-target \(D_j\).
\end{prop}

To clarify, here we work with the Gelfand triple \((D(A_j), \mathcal{H}^j, D(A_j)')\), and we redefine accordingly all the necessary operator spaces. Then \(T_j \in \mathcal{GL}_{\mathcal{H}^j}(D(A_j)')\), \(K_j \in \mathcal{L}(L_\lambda^j, \C)\), and equation \ref{eq:f-eq} becomes
\begin{equation}\label{eq:j-feq}
    \begin{cases}
        T_j(A_j + B_j K_j) = D_j T_j \ \text{in} \ \mathcal{L}(\mathcal{H}^j, D(A_j)'), \\
        T_j B_j = B_j \ \text{in} \ D(A_j)'.
    \end{cases}
\end{equation}
Notice that Proposition \ref{prop:partial-uniqueness} also applies here for \((A_j, B_j, D_j)\), and we denote by \(\Tilde{T}_j\), \(\tau_j\), and \(K_j\) the operators in \eqref{eq:TK-shape}. For the proof of Proposition \ref{prop:simple-feq}, we will need the following lemma.

\begin{lem}\label{lem:C_j}
     Let \(\mu \in \R_{> 0} \setminus \{\lambda_l - \lambda_h\}_{h \geq l \geq 1}\). There exists a parabolic \(F\)-equivalence of \((A_j, B_j, D_j)\) if and only if there exists \(C_j \in C({A_j}_H) \cap \mathcal{GL}_{\mathcal{H}_\lambda^j}(D(A_j)_\lambda')\) such that
     \begin{equation}\label{eq:TB=B}
         \tau_j {B_j}_L + C_j {B_j}_H = {B_j}_H.
     \end{equation}
\end{lem}

\begin{proof}
Let \((T_j, K_j)\) be a parabolic \(F\)-equivalence between \((A_j, B_j, D_j)\). By definition, we have \(T_j B_j = B_j\). Using Proposition \ref{prop:partial-uniqueness} for \((A_j, B_j, D_j)\), we obtain:
\begin{equation}\label{eq:TB=Biff}
    T_jB_j = B_j \iff \begin{pmatrix}
        \tilde{T}_j {B_j}_L = {B_j}_L \\
        \tau_j {B_j}_L + C_j {B_j}_H = {B_j}_H
    \end{pmatrix}.
\end{equation}
Proposition \ref{prop:partial-uniqueness} ensures that \(C_j \in C({A_j}_H)\), and we have \textcolor{black}{(recall that $D(A_{j})'_{\lambda}$ correspond to the projection on the high frequencies)}:
\begin{equation}
    \forall v \in D(A_j)'_\lambda, \quad T_j v = C_j v.
\end{equation}
Since \textcolor{black}{\(T_j \in \mathcal{GL}_{\mathcal{H}^j_{\lambda}}(D(A_j)_{\lambda}')\)}, 
the equality above implies \(C_j \in \mathcal{GL}_{\mathcal{H}_\lambda^j}(D(A_j)_\lambda')\).

Conversely, suppose that there exists \(C_j \in C({A_j}_H) \cap \mathcal{GL}_{\mathcal{H}_\lambda^j}(D(A_j)_\lambda')\) \textcolor{black}{such that $\tau_j {B_j}_L + C_j {B_j}_H = {B_j}_H$} holds. Then, we can define \(T_j\) and \(K_j\) as in Proposition \ref{prop:partial-uniqueness}. By Theorem \ref{thm:ffeq}, we \textcolor{black}{have} that \(\tilde{T}_j {B_j}_L = {B_j}_L\), so by Equation \eqref{eq:TB=Biff}, we have \(T_j B_j = B_j\).

Next, let's verify that \(T_j \in \mathcal{GL}_{\mathcal{H}^j}(D(A_j)')\). Define
\begin{equation}\label{eq:expr_inv}
    R_j = \begin{pmatrix}
        \tilde{T}_j^{-1} & 0 \\
        -C_j^{-1} \tau_j \tilde{T}_j^{-1} & C_j^{-1}
    \end{pmatrix}.
\end{equation}
Noticing that \(-C_j^{-1} \tau_j \tilde{T}_j^{-1} \in \mathcal{L}(L_\lambda^j, H_\lambda)\), and using the fact that \(C_j \in \mathcal{GL}_{\mathcal{H}_\lambda^j}(D(A_j)_\lambda')\), we have \(R_j \in \mathcal{L}_{\mathcal{H}^j}(D(A_j)')\). Then \(T_j \in \mathcal{GL}_{\mathcal{H}^j}(D(A_j)')\) immediately follows from the relation
\begin{equation}
    T_j R_j = R_j T_j = \text{Id}_{D(A_j)'}.
\end{equation}

We now need to demonstrate that the first equation in \eqref{eq:j-feq} holds. Let $n > N_j(\lambda)$, then we have 
\begin{align}
\begin{split}
    T_j(A_j + B_j K_j)e_n^j &= T_j A_j e_n^j \\
    &= C_j A_j e_n^j \\ 
    &= A_j C_j e_n^j \\
    &= D_j T_j e_n^j.
\end{split}
\end{align}

Next, consider the case where $n \leq N_j(\lambda)$. Using $T_j B_j = B_j$, we have 
\begin{align}
\begin{split}
    T_j(A_j + B_j K_j)e_n^j &= \lambda_n^j T_j e_n^j + {B_j}_L K_j e_n^j + {B_j}_H K_j e_n^j \\
    &= \underbrace{(\tilde{T}_j {A_j}_L + {B_j}_L K_j)e_n^j}_{\in L_\lambda^j} + \underbrace{(\lambda_n^j \tau_j + {B_j}_H K_j) e_n^j}_{\in D(A_j)'_\lambda}.
\end{split}
\end{align}

By Theorem \ref{thm:ffeq}, we have
\begin{equation}
    (\tilde{T}_j {A_j}_L + {B_j}_L K_j)e_n^j = ({A_j}_L - \lambda_n^j)\tilde{T}_j e_n^j.
\end{equation}

Then, by the definition of $\tau_j$, and setting $f_k^j := \sqrt{1 + |\lambda_k^j|^2} e_k^j$, we have 
\begin{equation}
    (\lambda_n^j \tau_j + {B_j}_H K_j) e_n^j = \sum_{k \geq 1} \lambda_k^j \frac{\langle {B_j}_H K_j e_n^j, f_k^j \rangle_{D(A_j)'}}{\lambda_k^j - \lambda_n^j} f_k^j = {A_j}_H \tau_j e_n^j.
\end{equation}

Thus
\begin{equation}
    T_j(A_j + B_j K_j)e_n^j = ({A_j}_L - \lambda_n^j)\tilde{T}_j e_n^j + {A_j}_H \tau_j e_n^j = D_j T_j e_n^j.
\end{equation}

Finally, by continuity and linearity, we have 
\begin{equation}
\forall x \in \mathcal{H}^j, \quad T_j(A_j + B_j K_j)x = D_j T_j x.     
\end{equation}

\end{proof}
The preceding lemma is very useful for constructing parabolic \(F\)-equivalences, as it shows that we only need to construct \(C_j\). Notice that
\begin{equation}
        \forall M \in \mathcal{L}_{\mathcal{H}^j_\lambda}(D(A_j)'_\lambda), \quad M \in C({A_j}_H) \iff \forall n > N_j(\lambda), \ C_j e_n^j \in \ker({A_j}_H - \lambda_n^j).
\end{equation}

Hence, it is natural to try \(C_j\) as a diagonal operator, which means 
\begin{equation}
        \forall n > N_j(\lambda), \ C_j e_n^j = c_n^j e_n^j.
\end{equation}
Then the condition \(C_j \in \mathcal{GL}_{\mathcal{H}_\lambda^j}(D(A_j)_\lambda')\) simply becomes 
\begin{equation}\label{eq:bounded}
        \exists c_1^j, c_2^j > 0, \ \text{such that} \ c_1^j \leq |c_n^j| \leq c_2^j.
\end{equation}

\textcolor{black}{Note, however, that it is not clear yet that there would exist a parabolic $F$-equivalence of \((A_j, B_j, D_j)\) with such a diagonal $C_{j}$ since it also has to satisfy \eqref{eq:TB=B}. In fact, we are going to show that not only it is possible to have a parabolic $F$-equivalence of \((A_j, B_j, D_j)\) with $C_{j}$, but in addition it is nearly always possible even without additional assumption of $B_{H}$. More precisely,} we define \(\Lambda_j\) to be the set of all \(\mu > 0\) for which there exists a parabolic \(F\)-equivalence \((T_j, K_j)\) of \((A_j, B_j, D_j)\)\footnote{Recall that \(D_j\) depends on \(\mu\).} with \(C_j\) being diagonal. The following proposition describes important topological properties of \(\Lambda_j\).

\begin{prop}\label{prop:dense-feq}
\(\Lambda_j\) is open and dense in \(\R_{>0} \setminus \{\lambda_l - \lambda_h\}_{h \geq l \geq 1}\), and \(\R_{>0} \setminus \Lambda_j\) is negligible.
\end{prop}

\begin{proof}
First, we begin by demonstrating the openness of \(\Lambda_j\). To avoid confusion, we will explicitly indicate every dependency on \(\mu\) in our notation throughout this proof. Let \(\mu \in \Lambda_j\), then there exists a parabolic \(F\)-equivalence \((T_j^\mu, K_j^\mu)\) of \((A_j, B_j, D_j^\mu)\). Let \(k > N_j(\lambda)\), and we set 
\begin{equation}
    f_k^j = \sqrt{1+|\lambda_k^j|^2} e_k^j, \quad b_k^j := \scalp{B_j}{f_k^j}_{D(A_j)'}, \quad K_n^j := K_j \textcolor{black}{f_n^j}.
\end{equation}
Then, projecting \eqref{eq:TB=B} onto \(\textcolor{black}{f_k^j}\), we obtain
\begin{equation}
    (1-{c_k^{\mu,j}})b_k^j = \sum_{n=1}^{N_j(\lambda)} b_n^j \scalp{\tau^\mu_j (\textcolor{black}{f_n^j})}{f_k^j}_{D(A_j)'} = b_k^j \sum_{n=1}^{N_j(\lambda)} \frac{b_n^j K_n^{\mu, j}}{\lambda_k^j - \lambda_n^j}.
\end{equation}
Hence, if \(b_k^j \neq 0\), \textcolor{black}{this imposes}
\begin{equation}
    c_k^{\mu,j} = 1 -  \sum_{n=1}^{N_j(\lambda)} \frac{b_n^j K_n^{\mu, j}}{\lambda_k^j - \lambda_n^j}.
\end{equation}
Now, let \(\delta \in \R\) such that \(\mu + \delta \in \R_{>0}  \setminus \{\lambda_l - \lambda_h\}_{h \geq l \geq 1}\). We will show that if \(|\delta|\) is small enough, then \(\mu + \delta \in \Lambda_j\), thereby proving that \(\Lambda_j\) is an open set of $\R_{>0}  \setminus \{\lambda_l - \lambda_h\}_{h \geq l \geq 1}$.

Let \(k > N_j(\lambda)\), and define \(C_j^{\mu + \delta}\) as follows:
\begin{itemize}
    \item[--]  If \(b_k^j = 0\), then set \(c_k^{\mu + \delta, j } = 1\).
    \item[--]  Otherwise, set \(\displaystyle c_k^{\mu + \delta, j } =  1 -  \sum_{n=1}^{N_j(\lambda)} \frac{b_n^j K_n^{\mu + \delta, j}}{\lambda_k^j - \lambda_n^j}\).
\end{itemize}

Now, by Lemma \ref{lem:C_j}, we only need to check that \eqref{eq:bounded} holds for \((c_k^{\mu +\delta, j})_{k > N(\lambda)}\). Notice that the above expressions imply, under Hypothesis \ref{item:A2}, that \textcolor{black}{(note that the number of terms in the sum is finite and does not depend on $k$)}
\begin{equation}
\label{eq:limc}
    \lim_{k \rightarrow +\infty} c_{k}^{\mu+\delta,j} = 1.
\end{equation}
Hence, the sequence is bounded from above. \textcolor{black}{Since $\mu\in \Lambda_{j}$, then \eqref{eq:bounded} holds for $(c_{k}^{\mu,j})_{k\geq N(\lambda)}$. Now observe that, from
Theorem \ref{thm:ffeq},} the $K_{n}^{\mu+\delta,j}$ are continuous in $\delta$. Also notice that the number of terms in the sum defining $c_k^{\mu + \delta,j}$ is finite and independent of $k$, and that we have (as $(\Re (\lambda^j_m))_{m \geq 1}$ is non-increasing)
\begin{equation}
    \forall k > N_j(\lambda), \ \forall n \leq N_j(\lambda), \ |\lambda_k^j - \lambda_n^j| \geq \Re (\lambda^j_{N_j(\lambda)+1} - \lambda_{N_j(\lambda)}^j) > 0.
\end{equation}
Thus, we deduce that the $c_k^{\mu + \delta,j}$ are continuous in $\delta$, uniformly in $k$. Therefore, \textcolor{black}{using \eqref{eq:bounded} with $(c_{k}^{\mu,j})_{k\geq N(\lambda)}$}
if \(|\delta|\) is small enough, we obtain that \textcolor{black}{there exists $c_{j}>0$ such that
\begin{equation}
    |c_{k}^{\mu+\delta,j}|>c_{j},\;\; \forall k> N(\lambda).
\end{equation}}
which shows that \( \mu + \delta \in \Lambda_j\).\\

Now we prove that \(\R_{>0} \setminus \Lambda_j\) is discrete in \(\R_{>0}\), hence countable, and this will demonstrate the last two assertions \textcolor{black}{of Proposition \ref{prop:dense-feq}}. We proceed by contradiction, suppose there exists \(\mu_{\infty} \in \R_{>0} \setminus \Lambda_j\) such that there exists an injective sequence \((\mu_m)_{m \geq 1}\) with \(\mu_m \in \R_{>0} \setminus \Lambda_j\) and \(\mu_m \rightarrow \mu_{\infty}\) as \(m \rightarrow \infty\). By the previous discussion and Lemma \ref{lem:C_j}, for each \(m \geq 1\), there exists \(k_m > N(\lambda)\) such that 
\begin{equation}\label{eq:c=1}
    c_{k_m}^{\mu_m,j} =  1 -  \sum_{n=1}^{N_j(\lambda)} \frac{b_n^j K_n^{\mu_m, j}}{\lambda_{k_m}^j - \lambda_n^j} = 0.
\end{equation}
Recall that, \textcolor{black}{from \eqref{eq:limc}, there exists $k_{0}$} large enough \textcolor{black}{such that} $c_k^{\mu_\infty,j} \geq 1/2$ for all $k > k_0$, and by the previous discussion, $c_k^{\mu_m,j}$ converges to $c_k^{\mu_\infty,j}$ uniformly in $k$, when $m \rightarrow +\infty$. Therefore, there 
\textcolor{black}{exists $m_{1}$ such that
\begin{equation}
    \forall m\geq m_{1}, \forall k>k_{0}, c_{k}^{\mu_{m},j}\neq 0.
\end{equation}
Therefore for $m\geq m_{1}$ there}
can only be a finite number of \(k\) where \eqref{eq:c=1} holds. Hence we can find \textcolor{black}{$k_{0}>N(\lambda)$ and} extract a subsequence $\psi$ \textcolor{black}{such that $k_{\psi_{m}}=k_{0}$, namely}
\begin{equation}
\label{eq:contra0}
 \ \forall m \geq 1, \ c_{k_0}^{\mu_{\psi(m)}, j} = 0.    
\end{equation}
However, by Theorem \ref{thm:ffeq}, we know that $K_{n}^{\mu,j}$, is a polynomial in \(\mu\) without constant term. Hence $c_{k_0}^{\mu,j}$ must be a non zero polynomial in $\mu$, but the isolated zero theorem and \eqref{eq:contra0} imply
\begin{equation}
   \forall \mu \in \R, \ c_{k_0}^{\mu,j} = 0,
\end{equation}
\textcolor{black}{which is a contradiction}.

\end{proof}

Finally, we can prove Proposition \ref{prop:simple-feq}.

\begin{proof}[Proof of Proposition \ref{prop:simple-feq}]
   We set \(\Lambda = \bigcap_{j \in \{1, \dots, m(\lambda)\}} \Lambda_j\). Proposition \ref{prop:dense-feq} ensures us that $\Lambda$ is full measure and dense in $\R_{>0}$, so let \(\mu > \lambda + c_A\) within $\Lambda$, then for all $j\in\{1, \dots, \ml\}$ there exists a parabolic $F$-equivalence of $(A_j,B_j,D_j)$.
   Now as \(\mu \geq \lambda + c_A\), we have
    \[
    \forall \nu \in \sigma(D_j), \ \nu \leq -\lambda.
    \]
    This implies that \(D_j\) is a \(\lambda\)-target. This concludes the proof of Proposition \ref{prop:simple-feq}.
\end{proof}

\subsubsection{Last step}\label{subsec:conc}
We now finalize the proof of our main result. By Proposition \ref{prop:simple-feq}, for almost every \(\mu \geq \lambda + c_A\) and for all \(j \in \{1, \dots, m(\lambda)\}\), there exists a parabolic \(F\)-equivalence \((T_j, K_j)\) between \((A_j, B_j)\) and \(D_j\). Hence, if we set
\begin{equation}
    T = T_1 + \dots + T_{m(\lambda)}, \quad K = (K_1, \dots, K_{m(\lambda)}),
\end{equation}
we have \(T \in \mathcal{GL}_H(D(A)')\) and \(K \in \mathcal{L}(L_\lambda, \C^{m(\lambda)})\).

Again, notice that
\begin{equation}
    D = D_1 + \dots + D_{m(\lambda)},
\end{equation}
and because \(\mu \geq \lambda + c_A\), we know that \(D\) is a \(\lambda\)-target. Finally, let \(x \in H\). We have \(x = x_1 + \dots + x_{m(\lambda)}\) with \(x_j \in \mathcal{H}^j\). By the definition of \((T_j, K_j)\), we have \(TB = B\) in \(D(A)'\), and
\begin{equation}
    T(A + BK)x = (TA + BK)x = \sum_{j=1}^{m(\lambda)} (T_j A_j + B_j K_j)x_j = \sum_{j=1}^{m(\lambda)} D_j T_j x_j = DTx.
\end{equation}
This concludes the proof of Theorem \ref{thm:main}.

\subsection{Stabilization of nonlinear systems}\label{subsec:stab-nonl-mainproof}
\textcolor{black}{In this subsection we prove Theorem \ref{thm:nonlinear-stab}. Let $A,\lambda,B$ as in Theorem \ref{thm:nonlinear-stab}, and let $(T,K)$ be an $F$-equivalence of $(A,B,D)$ given by Theorem \ref{thm:main}. As $
B$ is $F_\lambda$-admissible, without loss of generality we can suppose that $m=\ml$.}

\textcolor{black}{We start by showing the following essential lemma.
\begin{lem}\label{lem:iso-scale}
    Let $s\in[0,1]$ be such that $B\in (D_{-s}(A))^{\ml}$, then $T \in \mathcal{GL}(D_{r}(A))$ for all $r \in [-1,1-s]$.
\end{lem}
\begin{proof}
    By Proposition \ref{prop:partial-uniqueness}, we have, keeping the same notation
    \begin{equation}
        T =  \begin{pmatrix}
            \tilde{T} & 0 \\
            \tau & C
        \end{pmatrix}.
    \end{equation}
    Thus as $T \in \mathcal{GL}_H(D(A)')$ by hypothesis, using expression \eqref{eq:expr_inv} to construct an inverse, we only need to show that $\tau(L_\lambda) \subset D_{1-s}(A)$ to conclude. Using that $({\sqrt{1+|\lambda_k|^2}}^{r} e_k)_{k\geq 1}$ is a Hilbert basis of $D_{-r}(A)$ for all $r\in \R$, and Proposition \ref{prop:partial-uniqueness}, we have for all $n\leq N(\lambda)$
    \begin{equation}
        \tau(e_n) = \sum_{k\geq 1} \frac{\scalp{B_H Ke_n}{\sqrt{1+|\lambda_k|^2}^{s} }_{D_{-s}(A)} }{\lambda_k - \lambda_n} \sqrt{1+|\lambda_k|^2}^{s} e_k.
    \end{equation}
    As in the proof of Proposition \ref{prop:partial-uniqueness}, we set $I_k =\scalp{B_H Ke_n}{\sqrt{1+|\lambda_k|^2}^{s} }_{D_{-s}(A)}  $ which forms a $\ell^2$ sequence as $B_HKe_n \in D_{-s}(A)$, then we have
    \begin{equation}
        ||\tau(e_n)||_{D_{1-s}(A)}^2 = \sum_{k\geq 1 } \frac{|I_k|^2}{|\lambda_k-\lambda_n|^2} (1+|\lambda_k|^2) \lesssim \sum_{k\geq 1 }|I_k|^2 <   +\infty.
    \end{equation}
\end{proof}}

\textcolor{black}{
Now we set $\gamma= \min(1-s,\frac{1}{2})$
 with $s \in [0,1]$ such that $B \in (D_{-s}(A))^{\ml}$. By $F$-equivalence, $(T^{-1}(\frac{e_k}{\sqrt{1+|\lambda_k|^2}^\gamma}))_{k\geq1}$ is a Riesz basis of $D_\gamma(A+BK)$, and by Lemma \ref{lem:iso-scale} it is also a Riesz basis of $D_\gamma(A)$. Hence $D_\gamma(A+BK)=D_\gamma(A)$ and the two norms are equivalent, thus we have the following continuous inclusion 
 \begin{equation}\label{eq:inclusion_dom}
     D_{1/2}(A+BK) \hookrightarrow D_\gamma(A).
 \end{equation}
 }

Now notice that by $F$-equivalence, in the Hilbert space $H$ endowed with the norm $||T^{-1}\cdot||_H$, $A+BK$ satisfies \ref{item:Hyp1} in Appendix \ref{sec:Nonlin-ps} (it suffices to take $\mu>\lambda+c_A$ in Theorem \ref{thm:main}), and with this operator, $\mathcal{F}$ satisfies $\ref{item:Hyp2}$ thanks to \eqref{eq:inclusion_dom}. Hence applying Proposition \ref{prop:A-nonlin-wp}, we get the result of Theorem \ref{thm:nonlinear-stab} but for the norm $||T^{-1}\cdot||_H$ and with $C = 1$ in \eqref{eq:stab-nl}, then to conclude it suffices to notice that
\begin{equation}
    \forall x \in H, \quad ||T||_{\mathcal{L}(H)}^{-1} ||x||_H  \leq ||T^{-1}x||_H \leq ||T^{-1}||_{\mathcal{L}(H)} ||x||_H.
\end{equation}

\section{Approximate controllability and uniqueness}\label{section:wfeq}

The problem of finding an \(F\)-equivalence between a control system \((A, B)\) and a target \(D\) is challenging, especially when the problem is ill-posed, meaning there may be multiple pairs \((T, K)\) that satisfy the conditions. In this section, we investigate this issue of uniqueness.

First, in subsection \ref{subsec:wfeq}, we introduce an algebraic characterization for the lack of uniqueness on $T$. This allows us to introduce the weak $F$-equivalence formalism and to recover a well-posed problem. Then in subsection \ref{subsec:link-control}, we show that our algebraic characterization can be linked to the approximate controllability of $(A,B)$. This allows us to prove Theorem \ref{thm:approx-control-nbh}, which implies that the parabolic $F$-equivalence problem is well-posed if and only if $(A,B)$ is approximately controllable.

\subsection{Weak $F$-equivalence}\label{subsec:wfeq} 
Let us fix $A$, $\lambda$, $B$, $\mu$, and $D$ as in Theorem \ref{thm:main}.\footnote{Hence, with the notations of Section \ref{sec:MainProof}, $\mu \in \Lambda$.} Let $(T,K), (T',K') \in \mathcal{GL}_H(D(A)') \times \mathcal{L}(H, \C^{\ml})$ 
be parabolic $F$-equivalences of $(A,B,D)$. Proposition \ref{prop:partial-uniqueness} ensures that $K = K'$, and
\begin{equation}
    T_{LL} = T_{LL}', \ T_{LH} = T_{LH}', \ T_{HL} = T_{HL}', \ T_{HH}, T_{HH}' \in C(A_H).
\end{equation}
We adopt the same notation as in the proposition, hence we set $C := T_{HH}$ and $C' := T_{HH}'$. Now, notice that by the definition of $F$-equivalence, we have
\begin{equation}
    (T-T')B = B-B = 0.
\end{equation}
And by Lemma \ref{lem:C_j}, this gives us
\begin{equation}
    (C-C')B_H = 0.
\end{equation}
This leads us to define the following closed subspace of $C(A_H)$:
\begin{equation}
    N_{B_H} = \{M \in C(A_H) \, | \, MB_H = 0 \}.
\end{equation}
Then, we endow $\mathcal{L}_H(D(A)') / N_{B_H}$ with the quotient norm, and because $N_{B_H}$ is closed, the quotient is a Banach space. We denote by $\pi : \mathcal{L}_H(D(A)') \rightarrow \mathcal{L}_H(D(A)') / N_{B_H}$ the quotient map. Now, we have found all the possible solutions, \textcolor{black}{as illustrated by the following proposition.}
\begin{prop}\label{prop:sol-feq}
    Let $\mathcal{S} \subset \mathcal{GL}_H(D(A)') \times \mathcal{L}(H, \C^{\ml})$ be the set of all parabolic $F$-equivalences of $(A,B,D)$. We denote by $(T_*, K_*)$ \textcolor{black}{the} solution given by Theorem \ref{thm:main}. Then we have
    \begin{equation}
        \mathcal{S} = (\pi(T_*) \cap \mathcal{GL}_H(D(A)')) \times \{ K_* \}.
    \end{equation}
\end{prop}
\begin{proof}
    The above discussion shows that $\mathcal{S} \subset (\pi(T_*) \cap \mathcal{GL}_H(D(A)')) \times \{ K_* \}$ (as $\pi(T_*) = T_* + N_{B_H} $). Now let $T \in \mathcal{GL}_H(D(A)')$ such that there exists $N \in N_{B_H}$ with $T = T_* + N$. Hence, we have $TB = B$, and if we set $F(T) = T(A + BK) - DT$, using the same notation as in Proposition \ref{prop:partial-uniqueness}, we have
    \begin{equation}
    \label{eq:FT}
        F(T) = \begin{pmatrix}
            \tilde{T}(A_L + B_LK) - (A_L - \mu)\tilde{T} & 0 \\
            \tau (A_L + B_LK) + (C + N)B_HK - A_H \tau & (C + N)A_H - A_H(C + N)
        \end{pmatrix}.
    \end{equation}
    By definition, $NB_H = 0$, hence we have $F(T) = F(T_*) = 0$, and this concludes the proof.
\end{proof}

\textcolor{black}{We can deduce from this} a simple formalism that allows us to restate the parabolic $F$-equivalence problem so that it becomes well-posed. To this end let us fix $K \in \mathcal{L}(L_\lambda, \C^{\ml})$, if we set $\mathcal{T} := \pi(T)$, we first need to make sense of the following equation 
\begin{equation}\label{eq:weak-conjug}
    \mathcal{T}(A+BK) = D \mathcal{T}.
\end{equation}

We define $F_K : \mathcal{L}_H(D(A)') \rightarrow \mathcal{L}(H,D(A)')$ as in the previous proof, which is a bounded linear operator
\begin{equation}
    \forall T \in \mathcal{L}_H(D(A)'), \ F_K(T) = T(A+BK) - DT.
\end{equation}
Now for every $T \in \mathcal{L}_H(D(A)')$, we have
\begin{equation}
    F_K(T) = \begin{pmatrix}
        T_{LL}(A_L + B_LK) - (A_L - \mu)T_{LL} + T_{HL}B_HK & T_{HL}A_H - (A_L - \mu)T_{HL} \\
        T_{LH}(A_L + B_LK) + T_{HH}B_HK - A_H T_{LH} & T_{HH}A_H - A_H T_{HH}
    \end{pmatrix}.
\end{equation}
Now let $N \in N_{B_H}$. We have 
\begin{equation}
    F_K(N) = \begin{pmatrix}
        0 & 0 \\
        N B_H K & NA_H - A_H N
    \end{pmatrix} = 0.
\end{equation}
Hence, $F_K$ continuously factors through $\pi$, which means that there exists a unique bounded operator $\mathcal{F}_K$ from $\mathcal{L}_H(D(A)') / N_{B_H}$ to $\mathcal{L}(H,D(A)')$ such that
\begin{equation}
    \forall T \in \mathcal{L}_H(D(A)'), \ \mathcal{F}_K(\pi(T)) = F(T).
\end{equation}

Now this operator allows us to make sense of \eqref{eq:weak-conjug}, we simply say that $\mathcal{T}(A+BK)=D\mathcal{T}$ if $\mathcal{F}_K \mathcal{T} = 0$. Finally to define weak $F$-equivalence, we need to make sense of $\mathcal{T}B=B$, for this we define the following affine subspace  \begin{equation}
    \mathcal{F}_B = \{ T \in \L_H(D(A)') \, | \, TB = B \, \text{in} \, D(A)'\}.
\end{equation}
Then we define $\mathcal{T}B=B$ as $\mathcal{T} \in \pi(\mathcal{F}_B)$.

\begin{defn}[Weak \(F\)-equivalence]
Let \((\mathcal{T}, K) \in  \L_H(D(A)') / N_{B_H} \times \mathcal{L}(L_\lambda, \C^{\ml})\). We say that \((\mathcal{T}, K)\) is a weak \(F\)-equivalence of \((A, B, D)\), or that it is a weak \(F\)-equivalence between \((A, B)\) and \(D\), if 
\begin{equation}
    \mathcal{T} \in \pi(\mathcal{F}_B) \cap \ker \mathcal{F}_K.
\end{equation}
The above condition can also be written with the previous notations as
\begin{equation}\label{eq:weak-F-eq}
    \begin{cases}
        \mathcal{T}(A+BK)=D\mathcal{T}, \\ \mathcal{T}B=B.
    \end{cases}
\end{equation}
\end{defn}
Equation \eqref{eq:weak-F-eq} clearly explains why the previous definition is called a weak $F$-equivalence. Finally the next proposition shows that finding a weak $F$-equivalence is a well-posed problem and the solution is linked to parabolic $F$-equivalence.

\begin{prop}[Uniqueness of Weak $F$-Equivalence]
Let $(T_*, K_*)$ be a parabolic $F$-equivalence of $(A,B,D)$. Then $(\pi(T_*), K_*)$ is the unique weak $F$-equivalence of $(A,B,D)$.
\end{prop}
\begin{proof}
First, note that $(\pi(T_*), K_*)$ is indeed a weak $F$-equivalence. Now let $K\in \L(L_\lambda,\C^{\ml})$, $T \in \mathcal{L}_H(D(A)')$ such that $TB = B$ and 
\begin{equation}
    \pi(T)(A+BK) = D\pi(T).
\end{equation}
This implies that $\mathcal{F}_K(\pi(T)) = F_K(T) = 0$. The same reasoning as in the proof of Proposition \ref{prop:partial-uniqueness} shows that $K=K_*$ and $T - T_* \in C(A_H)$. Since $TB = B$ and $T_*B=B$, we have $(T-T_*)B=(T-T_*)_{HH}B=0$, hence $T-T_* \in N_{B_H}$, which gives $\pi(T) = \pi(T_*)$. 
\end{proof}

\subsection{Approximate controllability}\label{subsec:link-control} 
One might ask when finding a parabolic $F$-equivalence becomes a well-posed problem on its own. As we have shown in Proposition \ref{prop:sol-feq}, the issue of uniqueness is entirely due to the size of $N_{B_H}$, and having a unique solution is equivalent to $N_{B_H} = \{ 0 \}$. Therefore, in this subsection, we fix $A$, $\lambda$, $\mu$, and $D$ as before, but we let $B$ free.
\begin{defn}\label{defn:approx-control}
    Let $B \in (D(A)')^{\ml}$ \textcolor{black}{and $\tau>0$}, we say that $(A,B)$ is approximately controllable \textcolor{black}{(in time $\tau$)} if for all $u_0,u_1 \in H$ and any $\varepsilon > 0$, there exists $w \in L^2(0,\tau;\C^{\ml})$ such that the solution of the following system
    \begin{equation}
        \begin{cases}
            \partial_t u(t) = Au(t) + Bw(t), \ \forall t \in (0,\tau), \\
            u(0) = u_0,
        \end{cases}
    \end{equation}
    satisfies $\|u(\tau) - u_1\| \leq \varepsilon$.
\end{defn}
\begin{remark}
    By Lemma \ref{lem:Fattorini}, as for finite-dimensional systems, the approximate controllability of $(A,B)$ is in fact, independent of $\tau$. Hence, we will simply refer to the approximate controllability of $(A,B)$ without specifying a final time.
\end{remark}
Here our goal is to relate the size of $N_{B_H}$ to the approximate controllability of $(A,B)$ by proving the following theorem.
\begin{thm}\label{thm:approx-control-nbh}
    Let $B \in (D(A)')^{\ml}$ be $F_\lambda$-admissible. Then $N_{B_H} = \{0\}$ if and only if $(A,B)$ is approximately controllable.
\end{thm}
The above theorem immediately answers our question and provides a new characterization of approximate controllability for parabolic systems.
\begin{cor}
    Let $B \in (D(A)')^{\ml}$ be $F_\lambda$-admissible. Then $(A,B)$ is approximately controllable if and only if there exists a unique parabolic $F$-equivalence of $(A,B,D)$.
\end{cor}

To prove Theorem \ref{thm:approx-control-nbh}, we will use the generalized Fattorini criterion introduced by Badra and Takahashi in \cite{badra2014fattorini}. Let $B = (B_1, \dots, B_m) \in (D(A)')^{m}$,
since $A$ is normal, we have $\ker (A^* - \overline{\lambda_n}) = \ker(A - \lambda_n)$, hence $D(A)' = \bigoplus_{n \geq 1} \ker (A^* - \overline{\lambda_n})$. We set $l_n := \dim \ker(A - \lambda_n)$, which allows us to find a partition of $(e_n)_{n \geq 1}$ as follows.\footnote{Note that we will work with $(e_n)_{n \geq 1}$, which is orthogonal but not orthonormal in $D(A)'$.} For each $n \geq 1$, we denote by $(\varepsilon_k^n)_{1 \leq k \leq l_n}$ a basis of $\ker(A - \lambda_n)$ formed by the elements
of $(e_k)_{k \geq 1}$, hence $(\varepsilon_k^n)$ is a reordering of $(e_n)$. Now, for $j \in \{1, \dots, m\}$, we can write in $D(A)'$
\begin{equation}
    B_j = \sum_{n \geq 1} \sum_{k=1}^{l_n} b_k^{j,n} \varepsilon_k^n.
\end{equation}
With this, we can now state the criterion for approximate controllability.

\begin{lem}[Fattorini-Badra-Takahashi Criterion]\label{lem:Fattorini}
    Let $B = (B_1, \dots, B_{\ml}) \in (D(A)')^{\ml}$. The pair $(A,B)$ is approximately controllable (for any time $\tau > 0$) if and only if for every $n \geq 1$, $\textnormal{rank}(\mathcal{B}_n) = l_n$, where $\mathcal{B}_n$ is given by
    \begin{equation}
        \mathcal{B}_n = \begin{pmatrix}
            b_1^{1,n} & b_2^{1,n} & \dots & b_{l_n}^{1,n} \\
            \vdots & \vdots & & \vdots \\
            b_1^{{\ml},n} & b_2^{{\ml},n} & \dots & b_{l_n}^{{\ml},n}
        \end{pmatrix}.
    \end{equation}
\end{lem}
\begin{proof}
    This is a direct application of Theorem 1.3 (using Remark 2.1 which allows us to take $\gamma=1$) in \cite{badra2014fattorini}.
\end{proof}
With this criterion we are now able to prove our theorem.
\begin{proof}[Proof of Theorem \ref{thm:approx-control-nbh}]
   Let $C \in C(A_H)$. First we give a characterization of $CB_H = 0$ using a collection of infinite scalar linear systems.
   
   Notice that $CB_H = 0$ is equivalent to $C{B_j}_H = 0$ for all $j \in \{1,\dots, {\ml}\}$. Now, for $n > N(\lambda)$, if we denote by $P_n$ the orthogonal projection onto $\ker(A_n - \lambda_n)$, we have $CP_nB = P_nCB$ because $CA_H = A_HC$. Hence, we have the following characterization of $CB_H = 0$
    \begin{equation}\label{eq:CPNB}
        \forall n > N(\lambda), \forall j \in \{1,\dots, {\ml}\}, \ CP_n{B_j}_H = 0.
    \end{equation}
    Thus, for each $n > N(\lambda)$, we have a finite-dimensional linear system. For $n \geq 0$ and for every $k \in \{1,\dots, l_n\}$, we set $c^n_k := C\varepsilon_{k}^n \in \ker(A - \lambda_n)$. Then, using \eqref{eq:CPNB}, we have 
    \begin{equation}
        CB_H = 0 \iff \forall n > N(\lambda), \quad \begin{cases}
            b_1^{1,n} c_1^n + \dots + b_{l_n}^{1,n} c_{l_n}^n = 0, \\
            b_1^{2,n} c_1^n + \dots + b_{l_n}^{2,n} c_{l_n}^n = 0, \\
            \vdots \\
            b_1^{{\ml},n} c_1^n + \dots + b_{l_n}^{{\ml},n} c_{l_n}^n = 0.
        \end{cases} 
    \end{equation}
    Now, to obtain scalar linear equations, we fix $n > N(\lambda)$. Then for each $k \in \{1,\dots, l_n\}$, we set $x_k := (\langle c_1^n, \varepsilon_k^n \rangle, \ \dots \ , \langle c_{l_n}^n, \varepsilon_k^n \rangle)^T \in \mathbb{C}^{l_n}$, and thus we have
    
    \begin{equation}\label{eq:carac-sys}
        \begin{cases}
            b_1^{1,n} c_1^n + \dots + b_{l_n}^{1,n} c_{l_n}^n = 0, \\
            b_1^{2,n} c_1^n + \dots + b_{l_n}^{2,n} c_{l_n}^n = 0, \\
            \vdots \\
            b_1^{{\ml},n} c_1^n + \dots + b_{l_n}^{{\ml},n} c_{l_n}^n = 0.
        \end{cases} \iff \forall k \in \{1,\dots, l_n\}, \ \mathcal{B}_n x_k = 0.
    \end{equation}

Now, we show that $N_{B_H} = \{0\}$ is equivalent to the approximate controllability of $(A,B)$. First, suppose that $(A,B)$ is approximately controllable. Then, Lemma \ref{lem:Fattorini} ensures that for all $n > N(\lambda)$, we have $\ker \mathcal{B}_n = \{0\}$, and by the above discussion, this implies that $N_{B_H} = \{0\}$. Conversely, suppose that $N_{B_H} = \{0\}$.  First, we establish the controllability of $(A_L, B_L)$. Using the decomposition from Lemma \ref{lem:decomp1}, it suffices to show that $({A_j}_L, {B_j}_L )$ is controllable for all $j \in \{1, \dots, m(\lambda)\}$. This follows from \hyperref[hyp:B]{$(H_{B})$}, thanks to the Kalman criterion. Thus, the pair $(A_L, B_L)$ is controllable.
Moreover, the Fattorini criterion applies to finite-dimensional systems as well, where approximate controllability coincides with exact controllability, hence we have
\begin{equation}
\forall n \in \{1, \dots, N(\lambda)\}, \ \text{rank} (\mathcal{B}_n) = l_n.
\end{equation}
Now, suppose by contradiction that there exists $n > N(\lambda)$ such that $\text{rank} (\mathcal{B}_n) < l_n$. Hence, there exists $z \in \mathbb{C}^{l_n} \setminus \{0\}$ such that $\mathcal{B}_n z = 0$. With this, we can construct $C \in C(A_H)$ such that it is zero everywhere except on $\ker(A_n - \lambda_n)$, where we have 
\begin{equation}
    \forall k \in \{1,\dots, l_n\}, \ C\varepsilon_k^n = z_k \varepsilon_1^n.
\end{equation}
Then, using \eqref{eq:carac-sys} and the above discussion, we would have $CB_H = 0$, which is a contradiction. Hence, by Lemma \ref{lem:Fattorini}, $(A,B)$ is approximately controllable.

\end{proof}
\begin{remark}
Working with an operator $A$ such that $m(\lambda) \underset{\lambda \rightarrow +\infty}{\rightarrow} +\infty$,
implies that $(A,B)$ is not approximately controllable for any $B \in (D(A)')^{{\ml}}$. \textcolor{black}{Indeed,} if $(A,B)$ is approximately controllable, then Lemma \ref{lem:Fattorini} implies $\sup_{n \geq 1} l_n < + \infty$.
Hence, in this case, we know that the problem of finding a parabolic $F$-equivalence for $(A,B,D)$ is ill-posed.
 
\end{remark}

\section{\textcolor{black}{Proof of Lemma \ref{lem:FQHE}}}
\label{sec:QHE}

\textcolor{black}{In this section, we show Lemma \ref{lem:FQHE}. We see from \eqref{eq:ql-nl-defF}} that we can decompose $\mathcal{F}$ in two parts, first the bilinear one derived from $B(u,v) = \div(\tilde{D}(u) \nabla v)$, and the remaining part $\widetilde{f}(u)$. Hence it suffices to check that both parts satisfy Assumption \ref{item:F1} to conclude that $\mathcal{F}$ satisfies it too. In order to do this, we need the following technical lemmas.
\begin{lem}\label{lem:QL-NL-f}
    Let $u,v \in H^s(\Omega)$ and let $g \in C^{1,1}_{\text{loc}}(\R)$. \textcolor{black}{Then} $g(u), g(v)\in H^{s-\varepsilon/2}(\Omega)$ and there exists $C>0$ such that
    \begin{equation}\label{eq:ql-nl-f}
        \|g(u) - g(v)\|_{H^{s-\varepsilon/2}(\Omega)} \leq C \Phi(\|u\|_{H^s} + \|v\|_{H^s}) \|u - v\|_{H^s(\Omega)}\textcolor{black}{,}
    \end{equation} 
    \textcolor{black}{w}here $\Phi : \R_{\geq 0} \to \R_{\geq 0}$ is defined as follows
    \begin{equation}\label{eq:ql-nl-phi}
        \Phi(r) = C_s \sup_{|x| \leq C_s r} |g'(x)| + C_s r \sup_{|x|,|y| \leq C_s r, x \neq y} \frac{|g'(x) - g'(y)|}{|x-y|},
    \end{equation}
    where $C_s>0$ is linked to the Sobolev constant of the embedding $H^s(\Omega) \hookrightarrow C^{0,\varepsilon}(\overline{\Omega})$.
 \end{lem}
 \begin{proof}
    We show inequality \eqref{eq:ql-nl-f}. 
    Let $u,v \in H^s(\Omega)$, we will use the following equivalent norm for $w \in H^{1+\sigma}(\Omega)$ (where $\sigma \in (0,1)$)
    \begin{equation}
        \|w\|_{H^{1+\sigma}(\Omega)}^2 = \|w\|_{L^2}^2 + [\nabla w]_{H^{\sigma}}^2,
    \end{equation}
    where $[\cdot]_{H^{\sigma}}^2$ is the Gagliardo seminorm defined as
    \begin{equation}
        [\textcolor{black}{f}]_{H^{\sigma}}^2 = \int_{\Omega} \int_{\Omega} \frac{|\textcolor{black}{f}(x) - \textcolor{black}{f}(y)|^2}{|x-y|^{d+2\sigma}} \,dx\,dy.
    \end{equation}
    Recall that $H^s(\Omega)$ is an algebra here, and it is continuously embedded in $C^{0,\varepsilon}(\overline{\Omega})$.
    Now we set $M= \|u\|_{L^\infty} + \|v\|_{L^\infty}$, the $L^2$ part of the norm is easy to estimate, and gives
    \begin{equation}
        \|g(u)-g(v)\|_{L^2} \leq \sup_{|x|\leq M} |g'(x)| \, \|u-v\|_{L^2}.
    \end{equation}
    Then we set $w=u-v$, hence we can write 
    \begin{equation}\label{eq:ql-decomp-lin}
        \nabla g(u) - \nabla g(v) = g'(u) \nabla w + (g'(u)-g'(v)) \nabla v.
    \end{equation}
    Let $h \in C^{0,\varepsilon}(\overline{\Omega})$ and $z \in H^{\varepsilon/2}(\Omega)$, recall the classical estimate estimate
    \begin{equation}\label{eq:sob-hold-product}
        [h z]_{H^{\varepsilon/2}}^2 \leq  C_\varepsilon \|h\|_{L^\infty}^2 \|z\|_{H^{\varepsilon/2}}^2 + C_\varepsilon [h]_{C^{0,\varepsilon}}^2 \|z\|_{L^2}^2   ,
    \end{equation}
    where $[\cdot]_{C^{0,\varepsilon}}$ is the Hölder seminorm. \textcolor{black}{The proof of \eqref{eq:sob-hold-product} goes as follows. First, we write for all $x,y \in \Omega$
\begin{equation}
    |h(x)z(x)-h(y)z(y)|^2 \leq 2|h(x)(z(x)-z(y))|^2 + 2|z(y)(h(x)-h(y))|^2.
\end{equation}
Then, the first part of the previous inequality gives the bound $\|h\|_{L^\infty}^2 \|z\|_{H^{\varepsilon/2}}^2$.  
For the second part, using the fact that $h \in C^{0,\varepsilon}(\overline{\Omega})$ and $z \in H^{\varepsilon/2}(\Omega)$, we have
\begin{align}
    \int_\Omega \int_\Omega \frac{|z(y)(h(x)-h(y))|^2}{|x-y|^{d+\varepsilon}} \, dx \, dy
    &\leq [h]_{C^{0,\varepsilon}}^2 \int_\Omega |z(y)|^2 \int_{\Omega+\Omega} \frac{1}{|h|^{d-\varepsilon}} \, dh \, dy  \\
    &\leq C_\varepsilon \,[h]_{C^{0,\varepsilon}}^2 \|z\|_{L^2}^2. \notag
\end{align}
    } Now applying this estimate to each term of the right hand side of \eqref{eq:ql-decomp-lin} and using the continuous embedding $H^{s}(\Omega) \hookrightarrow C^{0,\varepsilon}(\overline{\Omega})$, allow us to conclude thanks to direct computations.
\end{proof}
\begin{remark}\label{rem:ql-nl-phi}
    Notice that as $g \in C^{1,1}_{\text{loc}}(\R)$, $\Phi$ is non-decreasing continuous at $0$, and if $g'(0)=0$, then $\Phi(0)=0$. However, $\Phi$ might not be locally Lipschitz continuous, and this is the reason we needed Assumption \ref{item:F1} to be stated this way
    (see \cite{badra2014fattorini} for comparison).
\end{remark}
 \begin{lem}\label{lem:QL-NL-D}
    Let $u \in H^s(\Omega)$ and let $h \in C^{2,1}_{\text{loc}}(\R)$ with $h(0)=0$. Then $h(u) \in H^s(\Omega)$ and there exists $C>0$ such that
    \begin{equation}\label{eq:ql-nl-D}
        \|h(u)\|_{H^{s}(\Omega)} \leq C \chi(\|u\|_{H^s}) \|u\|_{H^s(\Omega)}.
    \end{equation} 
    Where $\chi : \R_{\geq 0} \to \R_{> 0}$ is a non-decreasing function \textcolor{black}{that does not depend on $u$}.
 \end{lem}

    \begin{proof}
        \textcolor{black}{W}e use the same notation as in the previous lemma. Let $u \in H^s(\Omega)$, as \textcolor{black}{previously} the $L^2$ part of the norm is easy to estimate \textcolor{black}{given that $h(0)=0$}, hence we focus on the Gagliardo seminorm of $\nabla h(u) = h'(u) \nabla u$. As $h' \in C^{1,1}_{\text{loc}}(\R)$, 
        Lemma \ref{lem:QL-NL-f} implies that $h'(u) \in H^{s-\varepsilon/2}(\Omega)$. \textcolor{black}{Recall that $s=d/2+\varepsilon$ so that} $s-\varepsilon/2 > d/2$, Sobolev algebra properties give us that there exists $C_{\varepsilon,s}>0$ such that
        \begin{equation}
            \|h'(u) \nabla u \|_{H^{\varepsilon}(\Omega)} \leq C_{\varepsilon,s} \|h'(u)\|_{H^{s-\varepsilon/2}(\Omega)} \|\nabla u\|_{H^{\varepsilon}(\Omega)}.
        \end{equation}
        Then we conclude using \eqref{eq:ql-nl-f} and direct computations.
    \end{proof}

    Now we compile everything in order to show that $\mathcal{F}$ satisfies Assumption \ref{item:F1}. For the bilinear part, using the continuity of $\div$ from $H^{s}(\Omega)$ to $H^{s-1}(\Omega)$, and the Sobolev algebra properties, we have, for all $u,v \in H^{s+1}(\Omega)$,
    \begin{equation}
        \|\div(\widetilde{D}(u)\nabla v)\|_{H^{s-1}} \leq C \|\widetilde{D}(u) \nabla v\|_{H^s} \leq C' \|\widetilde{D}(u)\|_{H^{s}} \|\nabla v\|_{H^{s}},
    \end{equation}
    where we used again that $s=d/2+\varepsilon>d/2$.
    Then using Lemma \ref{lem:QL-NL-D}, we obtain
    \begin{equation}
        \|\div(\widetilde{D}(u)\nabla v)\|_{H^{s-1}} \leq C''\chi(\|u\|_{H^{s}}) \|u\|_{H^{s}} \| v\|_{H^{s+1}}.
    \end{equation}
{\color{black} 
Hence as $\chi$ is non-decreasing, thanks to Remark \ref{rem:NL-sufficient}, we see that the bilinear part of $\mathcal{F}$ satisfies Assumption \ref{item:F1} with $\Phi=Id_\R$.} For the remaining part, applying Lemma \ref{lem:QL-NL-f} to $\widetilde{f}$ and with Remark \ref{rem:ql-nl-phi}, we see that it also satisfies Assumption \ref{item:F1} with $\Phi$ defined in Lemma \ref{lem:QL-NL-f}. 
    Hence $\mathcal{F}$ satisfies Assumption \ref{item:F1}.

\section{Acknowledgements}
\textcolor{black}{The authors} would like to thank Thomas Buchholtzer and Tanguy Lourme for their support during the writing of this article. \textcolor{black}{This work was supported by the ANR-Tremplin StarPDE (ANR-24-ERCS-0010).}
\appendix
\section{Finite dimensional F-equivalence}
We consider a finite dimensional control system 
\begin{equation}
    \dot{x} = Ax + Bu,
\end{equation}
with $x \in \textcolor{black}{\C}^n, \, u \in \textcolor{black}{\C}, \,  A \in M_n(\textcolor{black}{\C})$ and $B \in \textcolor{black}{\C}^n$.  

Here is the main result of this section, the proof is direct by computations (see \cite[Theorem 4.1]{Coron2015} for a more general case).
\begin{thm}\label{thm:ffeq}
    Let $\lambda \in \R$ and $(A,B)$ be controllable, then there exists one and only one $(T,K) \in \gl(n,\textcolor{black}{\C}) \times \textcolor{black}{\C}^{1  \times n}$ such that 
    \begin{equation}\label{eq:ff-eq}
        \begin{cases} T(A+BK)=(A-\lambda )T, \\  
        TB=B.\end{cases}
    \end{equation}
    Furthermore, $K_1, \dots, K_n$ are polynomials in $\lambda$, and if $A$ is \textcolor{black}{diagonalizable}, they have no constant term.
\end{thm}

\section{Nonlinear parabolic systems}\label{sec:Nonlin-ps}
\textcolor{black}{
In this Appendix, we show well-posedness and stability results for nonlinear parabolic equations, with the following settings:}

\begin{enumerate}[label=(H\arabic*)]
    \item \label{item:Hyp1}  Let $A$ be a diagonal parabolic operator, i.e. it satisfies \ref{item:A1} (and we suppose that $(e_n)_{n\geq 1}$ is a Hilbert basis), \ref{item:A2}, \ref{item:A3}, and $\lambda >0$ such that 
\begin{equation*}
    \forall n \geq 1, \quad \Re(\lambda_n) < -\lambda.
\end{equation*}
\item \label{item:Hyp2} Let $\mathcal{F}$ be a map from $D_{1/2}(A)$ to $D_{-1/2}(A)$, such that there exists $\eta , K>0$ and $\Phi : \R_{\geq 0} \rightarrow \R_{\geq 0}$ non-decreasing continuous at 0 with $\Phi(0)=0$, which satisfies the following conditions: for all $u,v\in D_{\gamma}(A)$ with $\|u\|_{H}, \|v\|_H \leq \eta$, we have
\begin{align}
\|\mathcal{F}(u)\|_{D_{-1/2}(A)}
    &\le \Phi(\|u\|_{H}) \,\|u\|_{D_{1/2}(A)}, \tag{H2.1}\label{eq:H21}\\
\|\mathcal{F}(u) - \mathcal{F}(v)\|_{D_{-1/2}(A)}
    &\le K(\|u\|_{D_{1/2}(A)} + \|v\|_{D_{1/2}(A)})\,\|u-v\|_{H} \tag{H2.2}\label{eq:H22}\\
    &\quad + K\Phi(\|u\|_{H} + \|v\|_{H})\,\|u-v\|_{D_{1/2}(A)} \notag.
\end{align}
\end{enumerate}
Then we have the following proposition.
\begin{prop}\label{prop:A-nonlin-wp}
    There exists $\delta > 0$ such that for every $u_0 \in H$, if $\|u_0\|_H \leq \delta$, there exists a unique solution $u \in C^0_b([0,+\infty);H) \cap L^2((0,+\infty);D_{1/2}(A))$ to the following system
    \begin{equation}\label{eq:A-nonlinear-system}
        \begin{cases}
            \d_t u = Au + \mathcal{F}(u), \\
            u(0)=u_0.
        \end{cases}
    \end{equation}
    Moreover, the previous system is exponentially stable, more precisely 
    \begin{equation}
        \forall t \geq 0, \quad ||u(t)||_H \leq e^{-\lambda t}||u_0||_H.
    \end{equation}
\end{prop}
The proof is based on a fixed-point argument, in order to do it, we need to study the following inhomogeneous system
\begin{equation}\label{eq:A-inhom-cauchy}
    \begin{cases}
            \d_t u = Au + g(t), \\
            u(0)=u_0,
    \end{cases}
\end{equation}
where $g \in  L^2((0,+\infty); D_{-1/2}(A))$.
\begin{lem}\label{lem:inhom-cauchy}
    Let $g \in L^2((0,+\infty); D_{-1/2}(A))$, for every $u_0 \in H$, there exists a unique solution $u \in C^0_b([0,+\infty);H) \cap L^2((0,+\infty); D_{1/2}(A))$ to \eqref{eq:A-inhom-cauchy}. Moreover, there exists $C \geq 1 $ independent of $u_0,u$ and $g$ such that
    \begin{equation}
        ||u||_{C^0([0,+\infty);H)}+||u||_{L^2((0,+\infty); D_{1/2}(A))}\leq C(||u_0||_{H} + ||g||_{L^2((0,+\infty); D_{-1/2}(A))}).
    \end{equation}
\end{lem}
\begin{remark}
    In the special case where $A$ is a self-adjoint operator, this result can be found in \cite[Chapter 10]{Brezis2010} and is attributed to Lions.
\end{remark}
\begin{proof}
See \cite[Chapter 10]{Brezis2010} for a proof when $A$ is self-adjoint, in our general case, the proof adapts similarly.
\end{proof}

\textcolor{black}{
Now, we can use a fixed point argument to prove the first part of Proposition \ref{prop:A-nonlin-wp}.}

\begin{proof}[Proof]
\textcolor{black}{
First, we define $X = C^0_b([0,{+\infty});H) \cap L^2((0,{+\infty}); D_{1/2}(A))$, it defines a Banach space with the following norm
\begin{equation}
    \forall u \in X, \quad ||u||_{X} := ||u||_{ C^0([0,{+\infty});H)} + ||u||_{ L^2((0,{+\infty}); D_{1/2}(A))}.
\end{equation}
Let $u_0 \in H$ be such that  $||u_0||_H \leq \delta$ and \(\delta \in (0,\frac{\eta}{2C})\) to be fixed later on, where $C$ is the constant of Lemma \ref{lem:inhom-cauchy}, we set $\kappa = 2C \delta$ and we define the following closed subset
$B(\kappa) = \{ u \in X \ |\ ||u||_{X} \leq \kappa \}.$
Let us define the mapping $M$ which, to each $u \in B(\kappa)$, associates $w$, the solution to
\begin{equation}
    \begin{cases}
        \partial_t w = A w + \mathcal{F}(u(t)), \\
        w(0)=u_0.
    \end{cases}
\end{equation}
Let $u \in B(\kappa)$, hence for all  $t\geq 0$ we have $\|u(t)\|_H \leq \kappa$, thus by \eqref{eq:H21} we have
\begin{equation}
    ||\mathcal{F}(u)||_{L^2((0,{+\infty}); D_{-1/2}(A))}^2 \leq \Phi(\kappa)^2 \int_0^{+\infty} ||u(s)||_{D_{1/2}(A)}^2 \, ds \leq \kappa^2 \Phi(\kappa)^2.
\end{equation}}

Hence by Lemma \ref{lem:inhom-cauchy} we get 
\begin{equation}\label{eq:estimatenorms}
   ||M(u)||_{X} \leq C(||u_0||_H +||\mathcal{F}(u)||_{L^2((0,{+\infty});D_{-1/2}(A))}) \leq C\kappa (\frac{1}{2C}+ \Phi(\kappa)) .
\end{equation}
As $\Phi$ is continuous at 0, we can choose $\delta$ small enough so that $\Phi(\kappa)\leq \frac{1}{2C}$ holds, this ensures that $M$ is a mapping from $B(\kappa)$ to itself. Moreover, let $u_1,u_2 \in B(\kappa)$. By first applying Lemma~\ref{lem:inhom-cauchy} and then using \eqref{eq:H22}, we have
\begin{align}
    ||M(u_1)-M(u_2)||_{X}^2 &\leq 2C^2K^2(||u_1||_X^2 + ||u_2||_X^2 + \frac{1}{2}\Phi(2\kappa)^2)||u_1-u_2||_{X}^2 \\
    &\leq C^2 K^2 (4\kappa^2 + \Phi(2\kappa)^2) ||u_1-u_2||_{X}^2.
\end{align}
This shows that taking $\delta$ small enough ensures that \( M \) is a contraction on \( B(\kappa) \), allowing us to apply the Picard fixed-point theorem to conclude that \eqref{eq:A-nonlinear-system} is well-posed in $X$ for small initial data.
\end{proof}

\textcolor{black}{Now, let $u_0 \in H$, as $u$ is a mild solution to \eqref{eq:A-nonlinear-system} its weak derivative is $Du+\mathcal{F}(u)$, and as $u \in X$ and by the hypothesis on $\mathcal{F}$,
we have $\d_t u \in L^2((0,+\infty);D_{-1/2}(A))$. This implies by classical approximation argument that $||u(\cdot)||_H^2 \in W^{1,1}_{loc}(0,+\infty)$ and
\begin{equation}\label{eq:weak-deriv}
    \frac{d}{dt} ||v(t)||_{H}^2 = 2 \Re(\scalp{\partial_t v}{v}_{D_{-1/2}(A),D_{1/2}(A)}) \ \text{a.e on} \ (0,+\infty).
\end{equation}
We now proceed to demonstrate the stability part of Proposition \ref{prop:A-nonlin-wp}, which will end the proof.}
\begin{proof}
We keep the same notation as in the previous fixed-point argument. Let $u_0 \in H$ such that $||u_0||_H \leq \delta$,
we denote by $u$ the solution to \eqref{eq:A-nonlinear-system} with $u(0) = u_0$, hence by the previous proof we know that
\begin{equation}
    \forall t \geq 0, \quad ||u(t)||_H \leq \kappa.
\end{equation}
Thus, for almost every $t > 0$, we have thanks to \eqref{eq:weak-deriv} and  \eqref{eq:H21}
\begin{align}
    \frac{1}{2}\frac{d}{dt} ||u||_{H}^2 &=  \Re ( \scalp{\partial_t u}{u}_{D_{-1/2}(A),D_{1/2}(A)}) \\
    &=  \Re ( \scalp{Au + \mathcal{F}(u)}{u}_{D_{-1/2}(A),D_{1/2}(A)}) \\
    &\leq \Re (\scalp{Au}{u}_{D_{-1/2}(A),D_{1/2}(A)}) + \Phi(\kappa) ||u||_{D_{1/2}(A)}^2.
\end{align}
By definition, we have for almost every $t >0$
\begin{equation}
\label{eq:decompAuu}
    \Re (\scalp{Au}{u}_{D_{-1/2}(A),D_{1/2}(A)}) = -\sum_{n\geq 1} |\Re \lambda_n| |u_n(t)|^2.
\end{equation}
Now by Hypothesis \ref{item:Hyp1}, there exists $\varepsilon > 0$ such that $|\Re \lambda_n| \geq \lambda + \varepsilon$ for every $n \geq 1$, thus we have 
\begin{equation}
    \Re (\scalp{Au}{u}_{D_{-1/2}(A),D_{1/2}(A)}) \leq - (\lambda + \varepsilon) ||u||_{H}^2.
\end{equation}
On the other hand, by hypothesis \ref{item:A3} there exists $c >0$ such that $c |\lambda_n| \leq |\Re \lambda_n|$ holds for all $n \geq 1$. Thus we have, using again \eqref{eq:decompAuu},
\begin{equation}
     \Re (\scalp{Au}{u}_{D_{-1/2}(A),D_{1/2}(A)}) \leq - c ||u||_{D_{1/2}(A)}^2.
\end{equation}
Now we set $\alpha= \frac{\varepsilon}{\lambda+\varepsilon}$, then multiplying the two previous inequality by respectively $\alpha$ and $1-\alpha$, we get that for almost every $t > 0$
\begin{equation}
    \frac{1}{2}\frac{d}{dt} ||u||_{H}^2 \leq - \alpha c ||u||_{D_{1/2}(A)}^2 - {\lambda} ||u||_H^2 + \Phi(\kappa) ||u||^2_{D_{1/2}(A)}.
\end{equation}
Hence, if we again reduce $\delta$ such that $\Phi(\kappa) \leq \alpha c$ holds, we can apply a classical variant of Gronwall lemma, and as $u \in C^0([0,+\infty);H)$, we have
\begin{equation}
    \forall t \geq 0, \quad ||u(t)||_H^2 \leq e^{- 2 \lambda t} ||u_0||_H^2,
\end{equation}
which concludes the proof.
\end{proof}



\section{Proof of Proposition \ref{prop:sg}}
\label{app:proofstab}

\textcolor{black}{The proof of Proposition \ref{prop:sg} is relatively straightforward and we give it here for completeness}: 
since $A \in \mathcal{L}(H, D(A)')$, $B \in (D(A)')^{\textcolor{black}{m(\lambda)}}$ 
and $K \in \mathcal{L}(H, \C^{\ml})$, we have $A + BK \in \mathcal{L}(H, D(A)')$.

Now we view $A + BK$ as an unbounded operator on $H$ and define $D(A + BK) = \{ x \in H \ | \ (A + BK)x \in H \}$. Let's show that $D(A + BK) = T^{-1}(D(A))$.

Let $x \in D(A + BK)$. Then $(A + BK)x \in H$, and so equality (\ref{eq:f-eq}) implies $D T x \in H$, hence $x \in T^{-1}(D(A))$. Conversely, let $x \in T^{-1}(D(A))$. Then again by (\ref{eq:f-eq}), we get $T(A + BK)x \in H$, but now because $T \in \mathcal{GL}(D(A)')$ and $T_{|H} \in \mathcal{GL}(H)$, we have $(A + BK)x \in H$. We know that $D(A)$ is dense, then since $T$ is an isomorphism on $H$, this shows the density of $D(A + BK)$.

Now we set $S(0) = \mathrm{Id}_H$ and for $t > 0$, we define $S(t) : H \rightarrow H$ as
\begin{equation}
    \forall x \in H, \ S(t)x = T^{-1} e^{t D} Tx.
\end{equation}
Hence, $S := (S(t))_{t \in \R_{\geq 0}}$ is a differentiable semigroup with growth rate at most $-\lambda$ as $(e^{tD})_{t\in \R_{\geq 0}}$ is a differentiable semigroup with the same growth rate. 
Now we have to check that $A + BK$ is the infinitesimal generator of $S$. Let $x \in H$ such that the limit in $H$ as $t \rightarrow 0$ of $\frac{S(t)x - x}{t}$ exists. We have
\begin{equation}
    \frac{S(t)x - x}{t} = T^{-1} \frac{e^{t D} Tx - Tx}{t},
\end{equation}
then the existence of the limit is equivalent to $Tx \in D(A)$ and hence $x \in D(A + BK)$ by what we have shown before. The last part of the proposition is an immediate consequence of semigroup theory for evolution equations, \textcolor{black}{see for instance} \cite[Sec. 4.1]{pazy1983semigroups}

\bibliographystyle{plain}    
\bibliography{References}

\end{document}